\newif\ifpaper
\paperfalse

\ifpaper
	\documentclass[final,leqno,onefignum,onetabnum]{siamltex1213} % remove final to show overfull hboxes
\else
	\documentclass[11pt]{article}
	\usepackage[
	    style=alphabetic,
	    hyperref=true,
	    backend=bibtex,
	    doi=false,
	    url=false,
	    isbn=false,
	    backref=false,
	    firstinits=false,
	    sorting=none,
	    maxcitenames=3,
	    maxbibnames=100,
	    block=none]{biblatex}
	\usepackage{amsthm}
	\usepackage{graphicx}
	\usepackage[hmargin=1.0in,bottom=1.8in]{geometry}
	\usepackage[sc]{mathpazo} 
	\usepackage{multirow}
	\usepackage{bm}
	\usepackage{todonotes}
	\usepackage{color}
	\definecolor{darkblue}{rgb}{0.0,0.0,0.8}
	\definecolor{darkred}{rgb}{0.8,0.0,0.0}
	\usepackage[colorlinks=true,citecolor=darkblue,urlcolor=darkblue,linkcolor=darkred]{hyperref}
	\newbibmacro*{string+doiurlisbn}[1]{%
	  \iffieldundef{doi}{%
		\iffieldundef{url}{%
		  \iffieldundef{isbn}{%
		    \iffieldundef{issn}{#1}%
			  {\href{http://books.google.com/books?vid=ISSN\thefield{issn}}{#1}}%
		  }%
			{\href{http://books.google.com/books?vid=ISBN\thefield{isbn}}{#1}}}%
		{\href{\thefield{url}}{#1}}%
	  }%
	  {\href{http://dx.doi.org/\thefield{doi}}{#1}}%
	}%
	\DeclareFieldFormat*{title}{\usebibmacro{string+doiurlisbn}{#1\isdot}}

	\DeclareFieldFormat{edition}{#1\addthinspace{}ed.}%

	\renewbibmacro{in:}{}

	\AtEveryBibitem{\clearfield{number}}

\fi

\usepackage{amsmath,amssymb,amsfonts}
\usepackage[caption=false]{subfig}
\usepackage{upgreek}
\usepackage{booktabs}
\usepackage[T1]{fontenc}
\usepackage{enumitem}
\usepackage{algorithm}
\usepackage{algpseudocode}
\usepackage{calc}
\usepackage{enumitem}

\newtheorem{lem}{Lemma}
\newtheorem{prop}{Proposition}
\newtheorem{rem}{Remark}
\newtheorem{exa}{Example}
\newtheorem{thm}{Theorem}
\newtheorem{cor}{Corollary}

\newcommand{\sspcoeff}{{\mathcal C}}
\newcommand{\hFE}{h_\textup{FE}}
\newcommand{\hmax}{h_\textup{max}}
\newcommand{\cfl}{\nu}

\newcommand{\cflFE}{\cfl_{\textup{FE}}}
\newcommand{\order}{{\mathcal O}}

\newcommand{\rhoFE}{\varrho_{\textup{FE}}}
\newcommand{\N}{{\cal{N}}}
\newcommand{\Ss}{{\cal{S}}}
\newcommand{\Hh}{{\cal{H}}}

\def\defaultgraphic{\includegraphics[width=0.48\textwidth]{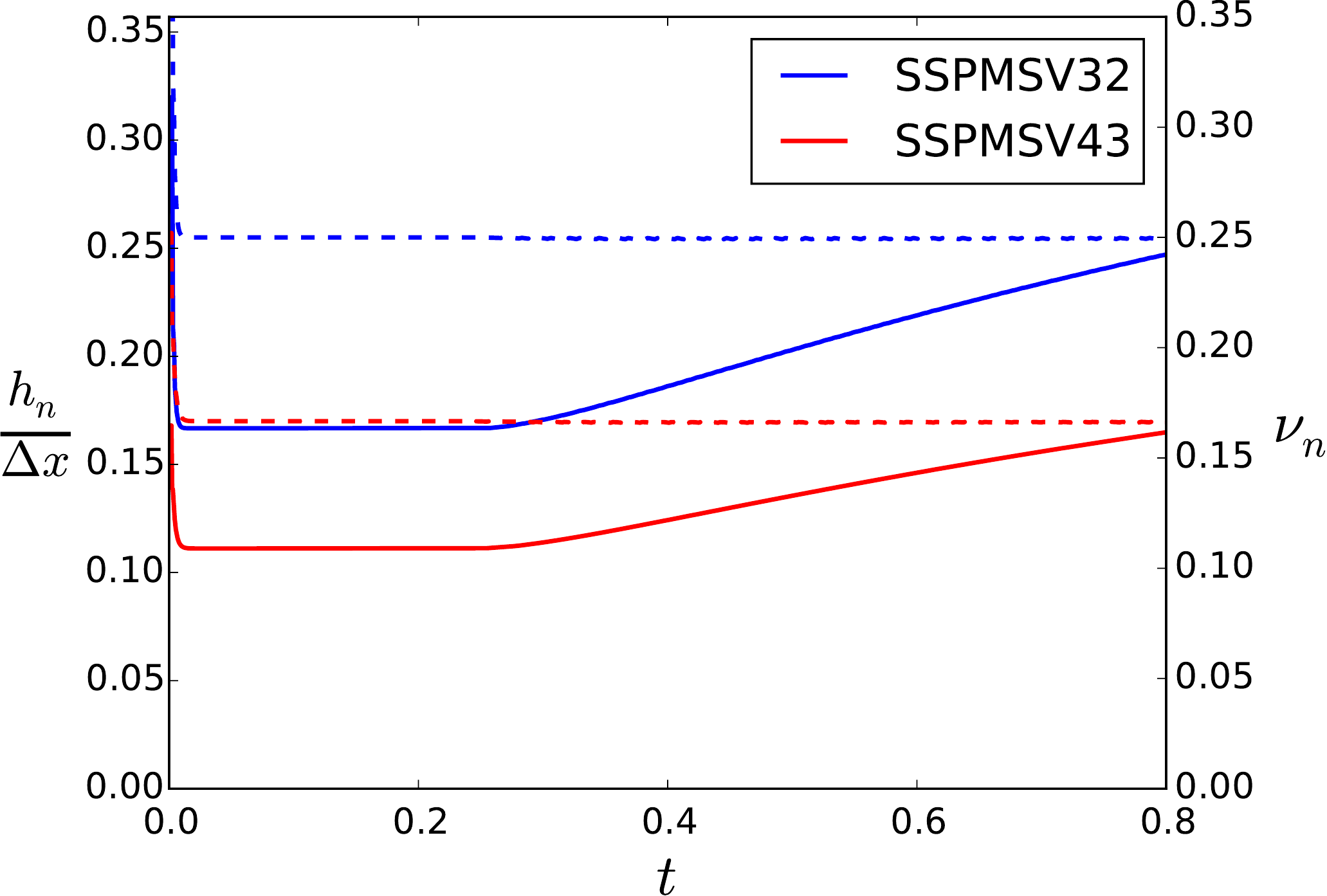}}
\newlength\graphicheight
\title{Strong stability preserving explicit linear multistep methods with variable step size}
\author{Yiannis Hadjimichael \and David I. Ketcheson\thanks{Corresponding author.} \and Lajos L{\'o}czi
\and Adri\'an N{\'e}meth\thanks{
    Authors listed alphabetically.    
    Author email addresses: {\texttt{\{yiannis.hadjimichael, david.ketcheson, 
lajos.loczi\}@kaust.edu.sa}},
{\texttt{nemetha@sze.hu}}. The fourth author was partially supported by the grant T\'AMOP-4.2.2.A-11/1/KONV-2012-0012. This work was supported by the King Abdullah University of Science and Technology (KAUST), 4700 Thuwal, 23955-6900, Saudi Arabia.}}

\begin{document}
\maketitle

\begin{abstract}
Strong stability preserving (SSP) methods are designed primarily for time
integration of nonlinear hyperbolic PDEs, for which the permissible
SSP step size varies from one step to the next.
We develop the first SSP linear multistep methods (of order two and three) with variable step
size, and prove their optimality, stability, and convergence.  
The choice of step size for multistep SSP methods is an
interesting problem because the allowable step size depends on the SSP coefficient,
which in turn depends on the chosen step sizes.
The description of the methods includes an optimal step-size strategy.
We prove sharp upper bounds on the allowable step size for explicit SSP linear multistep
methods and show the existence of methods with arbitrarily high
order of accuracy.
The effectiveness of the methods is demonstrated through numerical examples.
\end{abstract}

\ifpaper


	\pagestyle{myheadings}
	\thispagestyle{plain}
	\markboth{SSP EXPLICIT LMMs WITH VARIABLE STEP SIZE}{Y. HADJIMICHAEL, D. I. KETCHESON, L. L{\'O}CZI,
	A. N{\'E}METH}
	\slugger{sinum}{xxxx}{xx}{x}{x--x}%slugger should be set to mms, siap, sicomp, sicon, sidma, sima, simax, sinum, siopt, sisc, or sirev
\fi

\section{Introduction}
Strong stability preserving (SSP) linear multistep methods (LMMs) with uniform step size have been
studied by several authors
\cite{lenferink1989,lenferink1991,hundsdorfer2003,ruuth2005,hundsdorfer2005,ketcheson2009a}.
In this work, we develop the first variable step-size SSP multistep methods.

The principal area of application of strong stability preserving methods is
the integration of nonlinear systems of hyperbolic conservation laws.
In such applications, the allowable step size $\hmax$ is usually determined by a CFL-like
condition, and in particular is inversely proportional to the fastest wave
speed.  This wave speed may vary
significantly during the course of the integration, and its variation cannot generally
be predicted in advance.  Thus a fixed-step-size code may be inefficient
(if $\hmax$ increases) or may fail completely (if $\hmax$ decreases).

SSP Runge--Kutta methods are used much more widely than
SSP LMMs.  Indeed, it is difficult to find examples of SSP LMMs used in realistic
applications; this may be due to the lack of a variable step-size (VSS) formulation.
Tradeoffs between Runge--Kutta and linear multistep methods have been discussed at
length elsewhere, but one reason for preferring LMMs over their Runge--Kutta counterparts
in the context of strong stability preservation stems from
the recent development of a high-order positivity preserving limiter~\cite{zhang2010maximum}.
Use of Runge--Kutta methods in conjunction with the limiter can lead to 
order reduction, so linear multistep methods are recommended~\cite{Gottlieb2011a,zhang2010maximum}.

There exist two approaches to variable step-size multistep
methods \cite{hairer1993}.  In the first, polynomial interpolation is used
to find values at equally spaced points, and then the fixed-step-size method
is used.  In the second, the method coefficients are varied to maintain high
order accuracy based on the solution values at the given (non-uniform) step intervals.  Both
approaches are problematic for SSP methods; the first, because the interpolation
step itself may violate the SSP property, and the second, because the SSP
step size depends on the method coefficients.  Herein we pursue the second
strategy.

The main contributions of this work are:
\begin{itemize}
    \item sharp bounds on the SSP coefficient of variable step-size SSP LMMs (Theorems \ref{thm:upperbound}-\ref{thm:optimal3rdorder});
    \item optimal methods of orders two and three (Sections \ref{sec:2ndordermethods} and \ref{sec:3rdordermethods});
    \item existence of methods of arbitrary order (Theorem \ref{sec3Thm5});
    \item analysis of the greedy step-size strategy (Section \ref{sec:stepsize});
    \item proof of stability and convergence of the optimal methods (Theorems \ref{thm:stable}-\ref{ourconvergencethm}).
\end{itemize}

The rest of the paper is organized as follows. 
In Section \ref{sec:sspreview},
we review the theory of SSP LMMs while recognizing that the
forward Euler permissible step size may change from step to step.  The main
result, Theorem \ref{thm:ssplmm}, is a slight refinement of the standard one.
In Section \ref{sec:optimal} we show how an optimal SSP multistep formula may
be chosen at each step, given the sequence of previous steps.
In Section \ref{sec:full methods} we provide for convenience a description
of the two of the simplest and most useful methods in this work.
In Sections \ref{sec:2ndordermethods} and \ref{sec:3rdordermethods} we
derive and prove the optimality of several second- and third-order methods.
In Section \ref{sec:stepsize} we investigate the relation between the SSP step size, the
method coefficients, and the step-size sequence.  We develop step-size
strategies that ensure the SSP property under mild assumptions on the problem.
In Section \ref{sec:convergence} we prove that the methods, with the
prescribed step-size strategies, are stable and convergent.  In Section \ref{sec:examples}
we demonstrate the efficiency of the methods with some numerical examples.
Finally, Sections \ref{sec:proofs0} and \ref{sec:proofs} contain the proofs of
the more technical theorems and lemmas.

Two topics that might be pursued in the future based on this work are:
\begin{itemize}
    \item variable step-size SSP LMMs of order higher than three;
    \item variable step-size versions of SSP methods with multiple steps and multiple stages.
\end{itemize}

\section{SSP linear multistep methods\label{sec:ssplmm}}
We consider the numerical solution of the initial value problem
\begin{align} \label{ivp}
u'(t) & = f(u(t)) & u(t_0) = u_0,
\end{align}
for $t \in [t_0,t_0+T]$
by an explicit linear multistep method.  If a fixed numerical step size $h$ is
used, the method takes the form
\begin{align} \label{lmm-css}
    u_n & = \sum_{j=0}^{k-1} \left( \alpha_j u_{n-k+j} + h \beta_j f(u_{n-k+j}) \right) & n \ge k.
\end{align}
Here $k$ is the number of steps and $u_n$ is an approximation to the solution $u(n h)$.

Now let the step size vary from step to step so that $t_n=t_{n-1}+h_n$.
In order to achieve the same order of accuracy,
the coefficients $\alpha, \beta$ must also vary from step to step:
\begin{align} \label{lmm}
u_n = \sum_{j=0}^{k-1} \left( \alpha_{j,n} u_{n-k+j} + h_n \beta_{j,n} f(u_{n-k+j}) \right).
\end{align}

At this point, it is helpful to establish the following terminology.  We use the term
{\em multistep formula}, or just {\em formula}, to refer to a set of coefficients 
$\alpha_{j}$ ($=\alpha_{j,n}$), $\beta_{j}$ ($=\beta_{j,n}$)
that may be used at step $n$. 
 We use the term {\em multistep method} to refer to a full
time-stepping algorithm that includes a prescription of how to choose a formula $(\alpha_{j,n},\beta_{j,n})$
and the step size $h_n$ at step $n$. 
%Throughout the article, the acronyms LM, FSS and VSS
%stand for \textit{linear multistep}, \textit{fixed step size} and \textit{variable step size}, respectively.

For $1\le j \le k$ let
\begin{equation}\label{littleomegadef}
\omega_{j} := \frac{h_{n-k+j}}{h_n}>0
\end{equation}
denote the step-size ratios and
\begin{equation}\label{label13}
\left\{
\begin{aligned}
    \Omega_{0} := & 0,  & \\  
    \Omega_{j} :=  &  \sum_{i=1}^j \omega_{i} & \text{ for }\quad  & 1\le j\le k.
\end{aligned}
\right.
\end{equation}
Note that the values $\omega$ and $\Omega$ depend on $n$, but we often suppress that dependence
since we are considering a single step.\footnote{Our definition of $\omega$
    differs from the typical approach in the literature on variable step-size
    multistep methods, where only ratios of adjacent step sizes are used.  The
    present definition is more convenient in what follows.} 
It is useful to keep in mind that $\Omega_j = j$ if the step size is fixed.
Also the simple relation $\Omega_k=\Omega_{k-1}+1$ will often be used.

\subsection{Strong stability preservation} \label{sec:sspreview}
We are interested in initial value problems \eqref{ivp} whose solution
satisfies a monotonicity condition
\begin{align} \label{monotonicity}
\|u(t+h)\| & \le \|u(t)\| & \text{for } h \ge 0,
\end{align}
where $\|\cdot\|$ represents any convex functional (for instance, a norm).
We assume that $f$ satisfies the (stronger) forward Euler condition
\begin{align} \label{FE-condition}
\|u + h f(u)\| & \le \|u\| & \text{for } 0\le h \le \hFE(u).
\end{align}
The discrete monotonicity condition \eqref{FE-condition} implies the continuous
monotonicity condition \eqref{monotonicity}.

%In the literature on SSP methods, it is often assumed that 
%the value of $\hFE$ is independent of $u$.  
The primary application of SSP methods is in the time integration of nonlinear
hyperbolic PDEs.  In such applications, $\hFE$ is proportional to the CFL
number
\begin{align}
    \nu = h \frac{a(u)}{\Delta x}
\end{align}
where $a(u)$ is the largest wave speed appearing in the problem.  This speed
depends on $u$.  For instance, in the case of Burgers' equation
\begin{align}
    u_t + \left(\frac{u^2}{2}\right)_x & = 0,
\end{align}
we have $a(u) = \max_x |u|$.  For scalar conservation laws like Burgers' equation,
it is possible to determine a value of $\hFE$, based on
the initial and boundary data, that is valid for all time.  But for general systems of
conservation laws, $a(u)$ can grow in time and so the minimum value of $\hFE(u)$
cannot be determined without solving the initial value problem.
We will often write just $\hFE$ for brevity, but the dependence of $\hFE$ on
$u$ should be remembered.

We will develop linear multistep methods \eqref{lmm} that satisfy the discrete
monotonicity property
\begin{align} \label{discrete-monotonicity}
\|u_n\| \le \max(\|u_{n-k}\|,\|u_{n-k+1}\|,\dots,\|u_{n-1}\|).
\end{align}
The class of methods that satisfy \eqref{discrete-monotonicity}
whenever $f$ satisfies \eqref{FE-condition} are known as strong stability
preserving methods.  The most widely used SSP methods are one-step
(Runge--Kutta) methods.  When using an SSP multistep method, an SSP
Runge--Kutta method can be used to ensure monotonicity of the starting
values.  In the remainder of this work, we focus on conditions for monotonicity
of subsequent steps (for any given starting values).

The following theorem refines a well-known
result in the literature, by taking into account the dependence of $\hFE$ on $u$.
\begin{thm} \label{thm:ssplmm}
    Suppose that $f$ satisfies the forward Euler condition \eqref{FE-condition}
    and that the method \eqref{lmm} has non-negative coefficients $\alpha_{j,n}, \beta_{j,n} \ge 0$.
    Furthermore, suppose that the time step is chosen so that
    \begin{align}\label{ssp-step-sizefirst}
        0 \le h_n \le \min_{0 \le j \le k-1} \left(\frac{\alpha_{j,n}}{\beta_{j,n}}\, \hFE(u_{n-k+j})\right)
    \end{align}
    for each $n$, where the ratio $\alpha_{j,n}/\beta_{j,n}$ is understood as $+\infty$
    if $\beta_{j,n} = 0$.
    Then the solution of the initial value problem \eqref{ivp} given by the
    LMM \eqref{lmm} satisfies the monotonicity condition \eqref{discrete-monotonicity}.
\end{thm}

\begin{rem}
    The step-size restriction \eqref{ssp-step-sizefirst} is also {\em
    necessary} for monotonicity in the sense that, for any method \eqref{lmm},
    there exists some $f$ and starting values such that the monotonicity
    condition \eqref{discrete-monotonicity} will be violated if the step size
    \eqref{ssp-step-sizefirst} is exceeded.
\end{rem}

\begin{rem}\label{remark2} 
    Even in the case of the fixed-step-size method \eqref{lmm-css}, the theorem
    above generalizes results in the literature that are based on the assumption 
    of a constant $\hFE$.  It is natural to implement a step-size strategy
    that uses simply $\hFE(u_{n-1})$ in \eqref{ssp-step-sizefirst}, but this will not give
    the correct step size in general.  On the other hand, since $\hFE$ usually
    varies slowly from one step to the next, is often non-decreasing, and since
    the restriction \eqref{ssp-step-sizefirst}
    is often pessimistic, such a strategy will usually work well.
\end{rem}

Since $u_n$ (and hence $\hFE$) varies slowly from one step to the next,
it seems convenient to separate the factors in the upper bound in \eqref{ssp-step-sizefirst}
and consider the sufficient condition
\begin{align} \label{ssp-step-size}
    0 \le h_n \le \sspcoeff_n  \mu_n
\end{align}
where the SSP coefficient $\sspcoeff_n$ is
\begin{align}\label{sspcoeffdef}
    \sspcoeff_n & = \begin{cases} \max \left\{ r\in{\mathbb R^+} : \alpha_{j,n}-r\beta_{j,n} \ge 0\right\} & \text{ if } \alpha_{j,n} \ge 0, \  \beta_{j,n} \ge 0 \text{ for all $j$;} \\
0 & \text{otherwise,} \end{cases}
\end{align}
and
\begin{equation}\label{generalmundef}
\mu_n:=\min_{0 \le j \le k-1} \hFE(u_{n-k+j})\quad\quad (n\ge k).
\end{equation}
Note that in general the SSP coefficient varies from step to step, since it depends on the
method coefficients.

\subsection{Optimal SSP formulae\label{sec:optimal}}
For a given order $p$, number of steps $k$, and previous step-size
sequence $h_{n-1}, h_{n-2}, \dots$, we say that a multistep formula
is optimal if it gives the largest possible SSP coefficient $\sspcoeff_n$ in \eqref{ssp-step-size}
and satisfies the order conditions.  In this section we formulate
this optimization problem algebraically.
%In this section, we consider the choice of formula coefficients at step $n$ with $n\ge k$. 
%Thus, we assume that approximate solution values are given at $t_{n-k+j}$ for $1\le j \le k-1$
%and we seek to approximate the solution at $t_n$.
The linear multistep formula takes the form
\begin{align} \label{lmf}
u_n = \sum_{j=0}^{k-1} \left( \alpha_{j} u_{n-k+j} + h_n \beta_{j} f(u_{n-k+j}) \right).
\end{align}
Here we have omitted the subscript $n$ on the method coefficients to simplify
the notation.
The conditions for formula \eqref{lmf} to be consistent of order $p$ are
\begin{subequations} \label{oc-vss}
\begin{align}
    \sum_{j=0}^{k-1} \alpha_{j} & = 1 \\
    \sum_{j=0}^{k-1} \left(\Omega_{j}^m \alpha_{j} + m \Omega_{j}^{m-1} \beta_{j}\right) & = \Omega_{k}^m & 1 \le m \le p.
\end{align}
\end{subequations}
Let us change variables by introducing
\begin{align*}
    \delta_{j} := \alpha_{j} - r \beta_{j}.
\end{align*} 
Then the SSP coefficient of the formula is just the largest $r$ such that
all $\delta_j, \beta_j$ are non-negative \cite{lenferink1989,ketcheson2009a}.
In these variables, the order conditions \eqref{oc-vss} become
\begin{subequations} \label{oc-vss-delta}
\begin{align}
    \sum_{j=0}^{k-1} \left(\delta_{j} + r\beta_{j}\right) & = 1 \\
    \sum_{j=0}^{k-1} \left(\Omega_{j}^m (\delta_{j} + r\beta_{j}) + m \Omega_{j}^{m-1} \beta_{j}\right) & = \Omega_{k}^m & 1 \le m \le p. 
\end{align}
\end{subequations}
We will refer to a formula by the triplet $(\omega,\delta,\beta)$.
%Notice that \eqref{sspcoeffdef} implies for any $0\le r\le\sspcoeff$ that we have $\delta_j\ge 0$. 
Given $p$, $k$, and a set of step-size ratios $\omega_{j}$, the formula with the largest SSP coefficient
for the next step
 can be obtained by finding the largest $r$ such that
\eqref{oc-vss-delta} has a non-negative solution $\delta, \beta \ge 0$.
This could be done following the approach of \cite{ketcheson2009a}, by bisecting in $r$ and solving
a sequence of linear programming feasibility problems.  The solution of this optimization problem
would be the formula for use in the next step.  We do not pursue that approach here.
Instead, we derive families of formulae that can be applied based on the sequence
of previous step sizes.

%Since solving a sequence of linear programming problems at every time step seems cumbersome,
%in the rest of the paper we instead pursue a more traditional
%approach in which a prescribed family of formulae (with coefficients parametrized by the 
%values $\Omega_j$)
%is used throughout the computation.  We will show that
%in many cases 
%this is
%equivalent to the technique of finding and applying the optimal method at each step.

\subsection{Two optimal methods}\label{sec:full methods}
For convenience, here we list the methods most likely to be of interest
for practical application.  These (and other methods) are derived and analyzed
in the rest of the paper.  Recall that 
$\mu_n$ has been defined in \eqref{generalmundef}, and 
we assume $n\ge k$. The definition of the $\Omega$ quantities 
(with dependence on $n$ suppressed) is given in 
\eqref{littleomegadef}-\eqref{label13}.

\subsubsection{SSPMSV32}
Our three-step, second-order method is
%(compare \eqref{VSSSSPLMMk2} and \eqref{general-greedy} with $k=3$):
\[
u_n = \frac{\Omega_{2}^2-1}{\Omega_{2}^2} \left( u_{n-1} + \frac{\Omega_{2}}{\Omega_{2}-1}\, h_n \,f(u_{n-1})\right) + \frac{1}{\Omega_{2}^2}\, u_{n-3}
\]
with SSP step size restriction
\begin{equation*}
h_n \le \frac{h_{n-2}+h_{n-1}}{h_{n-2}+h_{n-1} + \mu_n}\cdot\mu_n.
\end{equation*}
If the step size is constant, this is equivalent to the known optimal
second-order three-step SSP method.
%\begin{rem}
%For any $k>3$, formulae \eqref{VSSSSPLMMk2} and \eqref{general-greedy} yield second-order methods with $k>3$ steps. 
%\end{rem}

\subsubsection{SSPMSV43}
Our four-step, third-order method is
%(compare \eqref{VSSSSPLMMk3} and \eqref{k45hndef} with $k=4$):
\begin{align*}
u_n = & \frac{(\Omega_{3}+1)^2 (\Omega_{3}-2)}{\Omega_{3}^3} \left( u_{n-1} + \frac{\Omega_{3}}{\Omega_{3}-2}\, h_n \,f(u_{n-1})\right) + \\
& \frac{3\Omega_{3}+2}{\Omega_{3}^3} \left( u_{n-4} + \frac{\Omega_{3}(\Omega_{3}+1)}{3\Omega_{3}+2}\, h_n \,f(u_{n-4})\right)
\end{align*}
with SSP step size restriction
\begin{equation*}
h_n \le \frac{\sum_{j=1}^{3} h_{n-j}}{\left(\sum_{j=1}^{3} h_{n-j} \right)+2\mu_n}\cdot\mu_n.
\end{equation*}
If the step size is constant, this is equivalent to the known optimal
third-order four-step SSP method.
%\begin{rem}
%Formulae \eqref{VSSSSPLMMk3} and
%\eqref{k45hndef} with $k=5$ yield a similar third-order method with five steps. 
%\end{rem}

\section{Existence and construction of optimal SSP formulae\label{sec:exandupper}}
In this section we consider the set of formulae satisfying \eqref{oc-vss-delta}
for fixed order $p$, number of steps $k$ and some step-size sequence $\Omega_j$.  
It is natural to ask whether any such formula exists,
what the supremum of achievable $r$ values is (i.e., the optimal SSP
coefficient $\sspcoeff$), and whether that supremum is attained by some formula. 
Here we give answers for certain classes.

In Section \ref{sec:upperboundon} we discuss how large an SSP coefficient can be,
and prove the existence of a formula with the maximum SSP coefficient.
In Sections \ref{sec:2ndordermethods} and \ref{sec:3rdordermethods}
we construct some practical optimal formulae of order 2 and 3, while 
the existence of higher-order formulae is established in Section \ref{sec:higherordermethods}.
The theorems of the present section are proved in Section \ref{sec:proofs0}. 
Our theorems are based on \cite{LMMpaper} by extending the corresponding results of 
that paper to the variable step-size case. The basic tools in \cite{LMMpaper} include Farkas' lemma, 
 the duality principle,  and the strong duality theorem of linear programming 
\cite{schrijver1998theory}.

\subsection{Upper bound on the SSP coefficient and existence of an optimal formula}\label{sec:upperboundon}
In the fixed-step-size case, the classical upper bound $\sspcoeff \leq \frac{k-p}{k-1}$
on the SSP coefficient $\sspcoeff$ for a $k$-step 
 explicit linear multistep formula of order $p$ with $k\ge p$ was proved in \cite{lenferink1989}
together with the existence of optimal methods. 
\begin{thm}\label{thm:upperbound}
 Suppose that some time-step ratios $\omega_j$ are given.   
Then the SSP coefficient for a $k$-step explicit linear multistep formula with order of accuracy $p\ge 2$
 is bounded by
\begin{equation}\label{boundonComegadeltabeta}
    \sspcoeff(\omega,\delta,\beta) \leq \begin{cases}
        0 & \text{if } \Omega_k \le p \\
        \frac{\Omega_{k}-p}{\Omega_{k}-1} & \text{if } \Omega_k>p.
    \end{cases}
\end{equation}
Moreover, suppose there exists a $k$-step explicit linear multistep formula of order $p\ge 2$ with positive SSP coefficient. 
Then there is a $k$-step formula of order $p$ whose SSP coefficient is equal to the optimal 
one.
\end{thm}

\subsection{Second-order formulae}\label{sec:2ndordermethods}

The bound in Theorem \ref{thm:upperbound} is sharp for $p=2$, as the following result shows.
\begin{thm}[Optimal second-order formulae with $k\ge 2$ steps] \label{thm:optimal2ndorder}
 Suppose that some time-step ratios $\omega_j$ are given.
Then there exists a second order linear multistep formula with $k$ steps and with positive SSP coefficient
 if and only if $\Omega_k > 2$. In this case, the optimal formula is 
    \begin{equation}\label{VSSSSPLMMk2}
    u_n = \frac{\Omega_{k-1}^2-1}{\Omega_{k-1}^2} \left( u_{n-1} 
                + \frac{\Omega_{k-1}}{\Omega_{k-1}-1}\, h_n \,f(u_{n-1})\right) 
                + \frac{1}{\Omega_{k-1}^2}\, u_{n-k},
    \end{equation}
 and has SSP coefficient
  \begin{equation}\label{eq17}
    \sspcoeff(\omega,\delta,\beta) = \frac{\Omega_k-2}{\Omega_k-1}.
  \end{equation}
\end{thm}

%Based on this theorem, we consider the order conditions
%\eqref{oc-vss-delta} up to $p=2$ and set all values to zero except for
%$r$, $\delta_0$, and $\beta_{k-1}$.  
%By using only 
%$\Omega_{k-1}> 0$ we see that the resulting system of
%three equations has a unique solution:
%\begin{subequations}\label{kstep2ndordermethodsformulae}
%\begin{align}
%r & = \frac{\Omega_{k-1} - 1}{\Omega_{k-1}}, \label{formula18a}\\
%\delta_0 & = \frac{1}{\Omega_{k-1}^2},\\
%\beta_{k-1} & = \frac{\Omega_{k-1}+1}{\Omega_{k-1}}.
%\end{align}
%\end{subequations}
%From this we obtain the optimal VSS formula
%Its SSP coefficient \eqref{formula18a} is positive if and only if 
%$\Omega_{k-1}>1$
%(being of course equivalent to \eqref{eq17} with $\Omega_{k}>2$).
%For any choice of step sizes we have $\delta_0$, $\beta_{k-1} >0$.
%Furthermore, 
%\begin{equation}
%r>0 \Longleftrightarrow\Omega_{k-1}>1.
%\end{equation}
\begin{rem}
If the step size is fixed, method \eqref{VSSSSPLMMk2} with $k\ge 3$ is equivalent to
the optimal $k$-step, second-order SSP method given in \cite[Section 8.2.1]{Gottlieb2011a}:
\begin{align} \label{sspk2}
u_n & = \frac{(k-1)^2-1}{(k-1)^2} \left( u_{n-1} + \frac{k-1}{k-2}\, h f(u_{n-1})\right) + 
\frac{1}{(k-1)^2}\, u_{n-k}
\end{align}
with SSP coefficient $\sspcoeff = (k-2)/(k-1)$.
\end{rem}

\subsection{Third-order formulae}\label{sec:3rdordermethods}

Compared to the family of second-order formulae above, the optimal third-order
formulae have a more complicated structure.  Although we will eventually focus
on two relatively simple 
third-order formulae (corresponding to $k=4$ and $k=5$), we present complete 
results 
in order to give a flavor of what may happen in the
search for optimal formulae.
The following Theorem \ref{thm:optimal3rdorder} characterizes optimal third-order linear multistep formulae and their SSP
coefficients, again for arbitrary step-size ratios $\omega_j$.
The theorem also provides an efficient way to find these optimal 
$k$-step formulae, since the sets of non-zero formula coefficients,
denoted by $\N\equiv\N(\omega,\delta,\beta)$, are
explicitly described. First we define $2k-2$ quantities and sets that will appear in Theorem \ref{thm:optimal3rdorder}.
\begin{enumerate}
\item[$\bullet$] For  $j=0:$ \[r_{0} := \frac{\Omega_k-3}{\Omega_k-1}, \quad 
\Ss_0:=\{\delta_0, \beta_0, \beta_{k-1}\}.\]
\item[$\bullet$] For $1 \leq j \leq k-2$:  \[r_{j} := \max\left(\frac{\Omega_k-\Omega_{j}-3}{\Omega_k-\Omega_{j}-1}, \frac{2}{\omega_j}+\frac{1}{\Omega_k-\Omega_{j-1}}\right),\]
 and either $\Ss_{j}:=\{\delta_j, \beta_j, \beta_{k-1}\}$  or $\Ss_{j}:=\{\delta_j, \beta_{j-1}, \beta_{j}\}$ or $\Ss_{j}:= \{\delta_j, \beta_{j-1}, \beta_j, \allowbreak \beta_{k-1}\}$, depending on whether the first expression is greater, or the second expression is greater, or the two expressions are equal in the above  
max(\ldots,\,\ldots).
\item[$\bullet$] For  $j=k-1$: 
\[r_{k-1} := \frac{2}{\omega_{k-1}} + \frac{1}{\Omega_k-\Omega_{k-2}},\quad 
\Ss_{k-1}:=\{\delta_{k-1}, \beta_{k-2}, \beta_{k-1}\}.\]
\item[$\bullet$] The $r_{k+j}$ quantities for $0\le j\le k-3$ are defined 
%in \eqref{rhokjdef}
below, and $\Ss_{k+j}:=\{\beta_{j}, \beta_{j+1}, \beta_{k-1}\}$. For any $0\le j\le k-3$ we set
\begin{equation}\label{Pkjdef}
P_{k+j}(x):=\Delta_j \Delta_{j+1} x^3- \left(\Delta_j \Delta_{j+1} +\Delta_j+\Delta_{j+1}\right) x^2+
2\left(\Delta _j+ \Delta _{j+1}+1\right) x-6,
\end{equation}
where 
\begin{equation}\label{Deltadef}
\Delta_{m}:=\Omega_k-\Omega_m.
\end{equation}
If
\begin{equation}\label{22cond}
\Delta_{j+1}^2-\left(\Delta_j+1\right) \Delta_{j+1}+3 \Delta_j>0 \quad \text{or}\quad \Delta _j<  5+2 \sqrt{6},
\end{equation}
then $P_{k+j}(\cdot)$ has a unique real root. We define
\begin{equation}\label{rhokjdef}
r_{k+j}:=
\begin{cases} 
\text{the real root of } P_{k+j}  & \text{ if } \eqref{22cond} \text{ holds }\\
+\infty & \text{ if } \eqref{22cond} \text{ does not hold.}
\end{cases}
\end{equation}
\end{enumerate}
%\begin{definition}\label{rkjdeflemma}
%For any $0\le j\le k-3$ we set
%\begin{equation}\label{Pkjdef}
%P_{k+j}(x):=\Delta_j \Delta_{j+1} x^3- \left(\Delta_j \Delta_{j+1} +\Delta_j+\Delta_{j+1}\right) x^2+
%2\left(\Delta _j+ \Delta _{j+1}+1\right) x-6,
%\end{equation}
%where  $\Delta_{m}:=\Omega_k-\Omega_m$.
%If
%\begin{equation}\label{22cond}
%\Delta_{j+1}^2-\left(\Delta_j+1\right) \Delta_{j+1}+3 \Delta_j>0 \quad \text{or}\quad \Delta _j<  5+2 \sqrt{6},
%\end{equation}
%then $P_{k+j}(\cdot)$ has a unique real root. We define
%\begin{equation}\label{rhokjdef}
%r_{k+j}:=
%\begin{cases} 
%\text{the real root of } P_{k+j}  & \text{ if } \eqref{22cond} \text{ holds }\\
%+\infty & \text{ if } \eqref{22cond} \text{ does not hold.}
%\end{cases}
%\end{equation}
%\end{definition}

\begin{thm}[Optimal third-order formulae with $k\ge 2$ steps] \label{thm:optimal3rdorder}
  Let time-step ratios $\omega_j$ be given.
  Then the inequality $\Omega_k > 3$ is necessary and sufficient for the existence of 
  a third-order, $k$-step
  explicit linear multistep formula with positive SSP coefficient.
For $\Omega_k > 3$, the optimal SSP coefficient is
  \begin{equation*}
    \sspcoeff(\omega,\delta,\beta) = \min\limits_{0 \leq j \leq 2k-3} r_{j},
  \end{equation*}
and the set of non-zero coefficients of an optimal SSP formula satisfies 
$\N\subseteq \Ss_\ell$, where the index $\ell\in\{0, 1, \ldots, 2k-3\}$
is determined by the relation $r_\ell=\min\limits_{0 \leq j \leq 2k-3} r_{j}$.
  If the index $\ell$ where the minimum is attained 
 is not unique and we have $\min\limits_{0 \leq j \leq 2k-3} r_{j}=r_{\ell_1}=\ldots=r_{\ell_m}$,
  then 
$\N\subseteq \Ss_{\ell_1}\cap \ldots\cap \Ss_{\ell_m}$.
\end{thm}

\begin{rem}
Let us highlight some differences between the fixed-step-size and the VSS cases 
concerning the sets of non-zero formula coefficients.\\
\indent 1. The pattern of non-zero coefficients for optimal third- or higher-order formulae
can be different from that of their fixed-step-size counterparts
 (this phenomenon does not occur in the class of optimal second-order formulae). A simple example is
provided by the optimal 3rd-order, 5-step formula with
  \[(\Omega_0, \ldots, \Omega_5):=\left( 0, 1, \frac{7}{3}, \frac{11}{3}, 5, 6 \right),\] 
where 
$\N=\{\beta_0, \beta_1, \beta_{4}\}$ 
(the coefficient pattern being similar to the case of the optimal fixed-step-size 3rd-order, 6-step method).\\
\indent 2.  If the index $\ell$ with $\sspcoeff(\omega,\delta,\beta)=r_\ell$ is not unique,
  then the optimal formula has less than 3 non-zero coefficients in the general case.\\
\indent 3.  If the index $\ell$ with $\sspcoeff(\omega,\delta,\beta)=r_\ell$ satisfies $1\le \ell \le k-2$ and the
  expressions in the $\max(\ldots)$ are equal, then the optimal formula is generally
  not unique and has more than 3 non-zero coefficients. For example, with
  $\omega_{j} := 5$  for $1\le j\le k-2$, $\omega_{k-1} := 4$ and $\omega_k:=1$,
we have a one-parameter family of optimal methods, and
$\N\subseteq \{\delta_{k-2}, \beta_{k-3}, \beta_{k-2}, \beta_{k-1}\}$ with
  \[\delta_{k-2}:=\frac{1}{7}(2-5\beta_{k-2}),\ 
   \beta_{k-3}:=\frac{2}{63}(16\beta_{k-2}-5) \text{ \ and\  }
   \beta_{k-1}:=\frac{5}{63}(20-\beta_{k-2})\]
for any ${5}/{16} \le \beta_{k-2}\le {2}/{5}$.
  However, for any fixed $\omega$ it can be shown that there is an optimal formula 
  that has at most $p$ non-zero coefficients just as in the fixed-step-size case \cite{lenferink1989}.
\end{rem}

% Statements for implicit methods removed as no implicit methods have been considered so far
% The $k$-step 1st-order implicit variable step-size LMM with optimal SSP radius has SSP coefficient $\mathcal{C}=\infty$, and the optimal
% method is the implicit Euler method.
%
% The $k$-step 2nd-order implicit variable step-size LMM with optimal SSP radius has SSP coefficient $\mathcal{C}=2$, and
% the optimal method has the same structure of non-zero $\beta_i,\delta_i$ coefficients as in the fixed step-size case.
%
% A $k$-step 3rd-order implicit variable step-size LMM can have a positive SSP coefficient iff $h_{n-1}+h_{n-2}+h_{n-3}>2h_{n}$.
% The method with the optimal radius of absolute monotonicity has SSP coefficient $\mathcal{C}=\frac{2\Omega_k-3}{\Omega_k-1}$ and
% has the same structure of non-zero $\beta_i,\delta_i$ coefficients as in the fixed step-size case.

The optimal fixed-step-size SSP method of order 3 and $k=4$ or $k=5$ steps has non-zero
coefficients $\{\delta_0, \beta_0, \beta_{k-1}\}$ \cite[Section 8.2.2]{Gottlieb2011a}.
In the rest of this section we consider formulae with 4 or 5 steps that generalize
the corresponding fixed-step-size methods.
A continuity argument shows that the set of non-zero coefficients 
is preserved if the step sizes are perturbed by a small enough amount.
Hence we solve the VSS order conditions \eqref{oc-vss-delta} 
with $p=3$ for $r$, $\delta_0$, $\beta_0$ and $\beta_{k-1}$:
by using $\Omega_{k-1}> 0$,
we obtain that the unique solution is
\begin{subequations}\label{kstep3rdordermethodformula}
\begin{align}
r & = \frac{\Omega_{k-1} - 2}{\Omega_{k-1}}, \label{r21aexpression}\\
    \delta_0 & = \frac{4(\Omega_{k-1}+1) - \Omega_{k-1}^2}{\Omega_{k-1}^3}, &
    \beta_{0} & = \frac{\Omega_{k-1}+1}{\Omega_{k-1}^2},&
\beta_{k-1} & = \frac{(\Omega_{k-1}+1)^2}{\Omega_{k-1}^2}.
\end{align}
\end{subequations}
The resulting VSS  formula reads
\begin{align}\label{VSSSSPLMMk3}
u_n   = & \frac{(\Omega_{k-1}+1)^2(\Omega_{k-1}-2)}{\Omega_{k-1}^3}\, u_{n-1} + \frac{(\Omega_{k-1}+1)^2}{\Omega_{k-1}^2}\, h_n f(u_{n-1}) + \nonumber\\
 & \frac{3\Omega_{k-1}+2}{\Omega_{k-1}^3}\, u_{n-k} + \frac{\Omega_{k-1}+1}{\Omega_{k-1}^2}\, h_n f(u_{n-k}).
\end{align}
\begin{prop}\label{prop1}  For 
\begin{equation}\label{Omegak-1range}
2 < \Omega_{k-1} \le 2(1+\sqrt{2})\approx 4.828,
\end{equation}
the SSP coefficient of \eqref{VSSSSPLMMk3} is optimal, and is equal to
\eqref{r21aexpression}.
%${(\Omega_{k-1} - 2)}/{\Omega_{k-1}}$.
\end{prop}
\begin{proof}
By using \eqref{kstep3rdordermethodformula}, formula \eqref{VSSSSPLMMk3} takes the form
\[
u_n = r \beta_{k-1} u_{n-1} + \beta_{k-1} h_n\,f(u_{n-1})
+ (r \beta_0+ \delta_0) u_{n-k} + \beta_0 h_n\,f(u_{n-k}).
\]
By definition (see \cite[Chapter 8]{Gottlieb2011a}), its SSP coefficient is given by
\begin{align*}
    \sspcoeff(\omega,\delta,\beta) = \min\left(\frac{r \beta_{k-1}}{\beta_{k-1}},\frac{r \beta_{0}+\delta_0}{\beta_{0}}\right)\equiv \min\left(\frac{\Omega_{k-1} - 2}{\Omega_{k-1}}, 
     \frac{3\Omega_{k-1} + 2}{\Omega_{k-1}(\Omega_{k-1}+1)} \right),
\end{align*}
from which we see that
\begin{align}\label{formula313}
    \sspcoeff(\omega,\delta,\beta) & = \begin{cases} 
    \frac{\Omega_{k-1} - 2}{\Omega_{k-1}} & \text{ for\ \  } 2 \le \Omega_{k-1} \le 2(1+\sqrt{2}) \\
    \frac{3\Omega_{k-1} + 2}{\Omega_{k-1}(\Omega_{k-1}+1)} & \text{ for\  \ } \Omega_{k-1} > 2(1+\sqrt{2}).
\end{cases}
\end{align}
%all formula coefficients \eqref{kstep3rdordermethodformula} are non-negative under condition
%\eqref{Omegak-1range}, hence the SSP coefficient of \eqref{VSSSSPLMMk3}
%is at least \eqref{r21aexpression}. 
But Theorem \ref{thm:upperbound} says that the 
SSP coefficient of any multistep formula with $p=3$ can be at most $\frac{\Omega_{k}-3}{\Omega_{k}-1}\equiv 
 \frac{\Omega_{k-1} - 2}{\Omega_{k-1}}$, so for $2 < \Omega_{k-1} \le 2(1+\sqrt{2})$ the 
 SSP coefficient of \eqref{VSSSSPLMMk3} is optimal. 
\end{proof}

The natural requirement \eqref{Omegak-1range} also justifies our choice for $k$:
  in the fixed-step-size case we have $\Omega_{k-1}=k-1$, and 
$2 < k-1 \le 2(1+\sqrt{2})$ holds if and only if $k=4$ or $k=5$.

\begin{rem} For $k\in\{4, 5\}$, we have the following strengthening of Proposition \ref{prop1}: the VSS formula \eqref{VSSSSPLMMk3}
is optimal if and only if \eqref{Omegak-1range} holds. 
To see this, it is enough to show that \eqref{VSSSSPLMMk3} is not optimal for
\begin{equation}\label{21+sqrt2}
\Omega_{k-1} > 2(1+\sqrt{2}).
\end{equation}
Indeed, by fixing any $\Omega_{k-1}> 2(1+\sqrt{2})$, one checks by direct computation that 
\begin{equation}\label{3omega+2}
\frac{3\Omega_{k-1} + 2}{\Omega_{k-1}(\Omega_{k-1}+1)} <  r_{j} \quad (j=0, 1, \ldots, 2k-3).
\end{equation}
But the SSP coefficient of \eqref{VSSSSPLMMk3} is given by the left-hand side of \eqref{3omega+2}
according to \eqref{formula313}, and the optimal SSP coefficient for third-order formulae is 
$\min\limits_{0 \leq j \leq 2k-3} r_{j}$ according to Theorem \ref{thm:optimal3rdorder}. Hence the SSP coefficient of \eqref{VSSSSPLMMk3} is not optimal when \eqref{21+sqrt2} holds.
\end{rem}
\begin{rem} One could develop optimal third-order explicit SSP formulae for $k>5$ as well. 
However, their structure, as indicated by Theorem \ref{thm:optimal3rdorder}, would be
more complicated, and the analysis performed in Section \ref{subsectionk3conv} would 
become increasingly involved.
\end{rem}

\subsection{Higher-order formulae}\label{sec:higherordermethods}

The last theorem in this section 
reveals that arbitrarily high-order VSS SSP
explicit linear multistep formulae exist,
though they may require a large number of steps.
\begin{thm}\label{sec3Thm5}
  Let $K_1, K_2\ge 1$ be arbitrary and let $p$ and $k > p^3K_1K_2/2$ be 
  arbitrary positive integers.
  Suppose that $\omega_{j}$ are given and that 
  \begin{equation}\label{K1K2omegajineq}
  1/K_1 \leq \omega_{j} \leq K_2\quad \text{ for all } 1\leq j\leq k.
  \end{equation}
  Then there exists a $k$-step formula of order $p$ with $\sspcoeff(\omega,\delta,\beta) > 0$.
\end{thm}

%The optimal 3-step, 2nd-order SSP LMM is
%$$
%u_n = \frac{3}{4} \left( u_{n-1} + 2h f(u_{n-1})\right) + \frac{1}{4} u_{n-3}.
%$$
%Here we extend this method to handle variable step sizes.
%The non-zero coefficients are $\delta_0, \beta_2$.  The equations
%for order two read
%\begin{subequations}
%\begin{align}
%\delta_0 + r \beta_2 & = 1 \\
%r\Omega_2 \beta_2 + \beta_2 & = \Omega_3 \\
%r \Omega_2^2 \beta_2 + 2 \Omega_2 \beta_2 & = \Omega_3^2.
%\end{align}
%\end{subequations}
%The unique solution is 
%\begin{subequations}
%\begin{align}
%r & = \frac{\Omega_2 -1}{\Omega_2} \\
%\beta_2 & = \frac{\Omega_2 + 1}{\Omega_2} \\
%\delta_0 & = \frac{1}{\Omega_2^2}.
%\end{align}
%\end{subequations}
%For any choice of step sizes, we have $\delta_0,\beta_2 >0$.
%However, if $\omega_1 + \omega_2 < 1$, then $r<0$.  This means
%that the step size cannot grow too quickly; we must choose 
%$h_n < h_{n-2} + h_{n-1}$ in order to have a positive step
%size restriction.

\section{Step-size selection and asymptotic behavior of the step sizes\label{sec:stepsize}}
To fully specify a method, we need not only a set of multistep formulae but also
a prescription for the step size.
When using a one-step SSP method to integrate a hyperbolic PDE, usually one chooses the step 
size $h_n:=\gamma\, \sspcoeff\, \hFE(u_{n-1})$, where $\gamma$ is
a safety factor slightly less than unity.
For SSP multistep methods, the choice of step size is
more complicated.  First, multiple previous steps must
be taken into account when determining an appropriate $\hFE$,
as already noted.  But more significantly, the SSP coefficient $\sspcoeff_n$
depends on the method coefficients, while the method 
coefficients depend on the choice of $h_n$.  These coupled
relations result in a step-size restriction that is a nonlinear 
function of recent step sizes.  In this section we propose a greedy step-size selection algorithm
and investigate the dynamics of the resulting step-size recursion for the formulae 
derived in the previous section.

Besides the step-size algorithms themselves, our main result will be that
the step size remains bounded away from zero, so the computation is guaranteed
to terminate.  Because the step-size sequence is given by a recursion involving
$\hFE$, we will at times require assumptions on $\hFE$:
\begin{align}\label{murestriction}
    \text{For all $n$ we have } \mu^- \le \hFE(u_n) \le \mu^+ \text{ for some } \mu^\pm \in (0,\infty).
\end{align}
\begin{align}\label{hferestriction}
    \text{For all $n$ we have }
        \rhoFE\le\frac{\hFE(u_n)}{\hFE(u_{n+1})}\le \frac{1}{\rhoFE}
    \text{ for some prescribed value } \rhoFE\in (0,1].
\end{align}
Assumption \eqref{murestriction} states that the forward Euler permissible step size
remains bounded and is also bounded away from zero.  For stable hyperbolic PDE discretizations,
this is very reasonable since it means that the maximum wave speed remains finite and non-zero.
Assumption \eqref{hferestriction} states that the forward Euler step size changes little
over a single numerical time step.  Typically, this is reasonable since it is a
necessary condition for the numerical solution to be accurate.  It can easily be checked
{\em a posteriori}.

\subsection{Second-order methods}\label{subsectionk2conv}
Let us first analyze the three-step, second-order method in detail.
%(since we will use this method in Section~\ref{sec:examples}), 
%but formulate our results in Theorem~\ref{hnconvergencetheorem}
%below for any $k\ge 3$.

%The optimal fixed-step-size 3-step 2nd-order SSP LMM is
%\[
%u_n = \frac{3}{4} \left( u_{n-1} + 2h f(u_{n-1})\right) + \frac{1}{4} u_{n-3},
%\]
Set $k=3$ in the second-order formula \eqref{VSSSSPLMMk2},
%gives 
%\[
%u_n = \frac{\Omega_{2,n}^2-1}{\Omega_{2,n}^2} \left( u_{n-1} + \frac{\Omega_{2,n}}%{\Omega_{2,n}-1}\, h_n \,f(u_{n-1})\right) + \frac{1}{\Omega_{2,n}^2}\, u_{n-3}.
%\]
%We refer to this method as \textbf{SSPMSV32}.
%For any choice of step sizes, we have $\delta_{0,n}, \beta_{2,n}>0$.
%First we collect some restrictions on the step-size sequence. 
and suppose that $h_j>0$ has already been defined for $1\le j\le n-1$ with some $n\ge 3$.
%Defining
%\begin{equation}\label{mundef}
%\mu_n:=\min_{0 \le j \le 2} \hFE(u_{n-3+j}),% \text{for } n\ge 3,
%\end{equation}
The SSP step-size restriction \eqref{ssp-step-size}
% becomes
%\begin{align} \label{ssr}
%    h_n \le \sspcoeff_n  \mu_n.
%\end{align}
is implicit, since $\sspcoeff_n$ depends on $h_n$.
By \eqref{eq17} we have
%\begin{equation}\label{hncnmun}
\[
\sspcoeff_n = \frac{\omega_{1,n} + \omega_{2,n} - 1}{\omega_{1,n} + \omega_{2,n}} \equiv \frac{h_{n-2} + h_{n-1} - h_n}{h_{n-2} + h_{n-1}}.
%    & = \frac{\omega_{1,n} + \omega_{2,n} - 1}{\omega_{1,n} + \omega_{2,n}} \cdot\mu_n \\
%   & = \frac{h_{n-2} + h_{n-1} - h_n}{h_{n-2} + h_{n-1}} \cdot \mu_n
%\end{align*}
%\end{subequations}
\]
%\end{equation}
Solving for $h_n$ in 
%\eqref{ssr} 
\eqref{ssp-step-size} gives
%\begin{equation}\label{ineq24}
\[
h_n \le \frac{h_{n-2}+h_{n-1}}{h_{n-2}+h_{n-1} + \mu_n} \cdot\mu_n.
\]
%\end{equation}
It is natural to take the largest allowed step size, i.e., to define
\begin{equation}\label{greedydef}
h_n:=\frac{h_{n-2}+h_{n-1}}{h_{n-2}+h_{n-1} + \mu_n}\cdot\mu_n.
\end{equation}
For the general $k$-step, second-order formula \eqref{VSSSSPLMMk2}, the same analysis
leads to the following choice of step size, which guarantees monotonicity:
\begin{align} \label{general-greedy}
h_n:=\frac{\sum_{j=1}^{k-1} h_{n-j}}{\left(\sum_{j=1}^{k-1} h_{n-j} \right)+\mu_n}\cdot\mu_n.
\end{align}
Note that this definition automatically ensures $\Omega_{k-1, n}>1$, and hence $\sspcoeff_n>0$
for any $\mu_n>0$.

\subsubsection{Asymptotic behavior of the step size}
Since \eqref{general-greedy} is a nonlinear recurrence, one might wonder if the
step size could be driven to zero, preventing termination of the integration process.
The following theorem shows that, under some natural assumptions, this cannot happen.
\begin{thm}\label{hnconvergencetheorem}
Consider the solution of \eqref{ivp} by the second-order formula \eqref{VSSSSPLMMk2} with some $k\ge 3$.
Let the initial $k-1$ step sizes be
positive and let the subsequent step sizes $h_n$ be chosen according to \eqref{general-greedy}.
Assume that \eqref{murestriction} holds with some constants $\mu^\pm$.
Then the step-size sequence $h_n$ satisfies
\begin{equation}\label{2ndorderinfsup}
\frac{k-2}{k-1}\,\mu^-\le \liminf_{n\to+\infty} h_n \le \limsup_{n\to+\infty} h_n \le \frac{k-2}{k-1}\,\mu^+.
\end{equation}
%In particular, no subsequence of $h_n$ can converge to $0$. 
As a special case, %if $\mu_n\equiv \mu^\pm\equiv \hFE>0$ for all $n\ge k$, that is, 
if $\hFE(u_{n})$ is constant, then 
\[
h_n \to \frac{k-2}{k-1}\,\hFE\quad (n\to +\infty).
\]
\end{thm}
\begin{rem}
The asymptotic step size $\frac{k-2}{k-1}\,\hFE$ given above is precisely
the allowable step size for the fixed-step-size SSP method of $k$ steps.
\end{rem}

\smallskip
The proof of Theorem \ref{hnconvergencetheorem} is given in Section \ref{sec:proofs}.
\smallskip

\begin{rem}\label{greedyoptimalremark}
Our greedy step-size selection \eqref{general-greedy} for second-order methods is 
optimal in the following sense. 
Let us assume that there is another step-size sequence, say $h_n^-$, with the following
properties:\\
\indent $\bullet$ the corresponding starting values are equal, that is, $h_j^-=h_j>0$ for $j=1, 2, \ldots, k-1$;\\
\indent $\bullet$ the $\mu_n$ quantities in \eqref{generalmundef} 
corresponding to the sequences $h_n$ and $h_n^-$ are all equal to a fixed common constant $\mu>0$;\\
\indent $\bullet$ $h_n^-$ satisfies \eqref{general-greedy} with inequality, that is,
\[
h_n^- \le \frac{\sum_{j=1}^{k-1} h_{n-j}^-}{\left(\sum_{j=1}^{k-1} h_{n-j}^-\right) + \mu} \cdot\mu.
\]
Then---as a straightforward modification of the proof of \eqref{hnsandwich} in 
Section \ref{proofofTh2section} shows---we have $h_n^-\le h_n$ for all $n\ge 1$. 
\end{rem}

\subsection{Third-order methods}\label{subsectionk3conv}
In this section we give our step-size selection algorithm for the 3rd-order $4$-step and $5$-step 
SSP formulae \eqref{VSSSSPLMMk3} by following the same approach as in Section \ref{subsectionk2conv}.

By using the optimal SSP coefficient 
$\sspcoeff_n=\frac{\Omega_{k-1, n} - 2}{\Omega_{k-1, n}}$ 
given in Proposition \ref{prop1} we see 
%after some elementary manipulations 
that $h_n=\sspcoeff_n  \mu_n$ 
%equal to the maximal SSP step size 
%$\sspcoeff_n  \mu_n$ 
in \eqref{ssp-step-size} if and only if 
%\begin{align} \label{general-greedy-3}
%    h_n=\begin{cases} 
%        \frac{\sum_{j=1}^{k-1} h_{n-j}}{\left(\sum_{j=1}^{k-1} h_{n-j} \right)+2\mu_n}\cdot\mu_n & \text{if } \sum_{j=1}^{k-1} h_{n-j} \le \sqrt{8}\, \mu_n, \\ %2 \le \Omega_{k-1} \le 2(1+\sqrt{2}) \\
%        \frac{\left(3\mu_n-\sum_{j=1}^{k-1} h_{n-j}\right)\left(\sum_{j=1}^{k-1} h_{n-j}\right)}{\left(\sum_{j=1}^{k-1} h_{n-j}\right)-2\mu_n } & \text{if } \sqrt{8}\,\mu_n<\sum_{j=1}^{k-1} h_{n-j} < 3\mu_n.%\Omega_{k-1} > 2(1+\sqrt{2}).
%\end{cases}
%\end{align}
\[
h_n= \frac{\sum_{j=1}^{k-1} h_{n-j}}{\left(\sum_{j=1}^{k-1} h_{n-j} \right)+2\mu_n}\cdot\mu_n.
\]
This relation also yields that
\[
\Omega_{k-1, n} = 2+\frac{1}{\mu_n}\sum_{j=1}^{k-1} h_{n-j},
\]
so $\Omega_{k-1, n}>2$ is guaranteed by $\mu_n>0$. 
Therefore, \eqref{Omegak-1range} is
equivalent to
\begin{equation}\label{sqrt8muineq}
\sum_{j=1}^{k-1} h_{n-j} \le \sqrt{8}\, \mu_n.
\end{equation}
The definition of $h_n$ in Theorem \ref{hnconvergencetheorem3rdorder} below is based on these considerations. 
The assumptions of the theorem
on the starting values 
%$h_j$ ($1\le j\le k-1$), and
and on the problem (involving the boundedness of the $\hFE$ quantities and that of their ratios)
% of  
%successive $\hFE$ quantities---these inequalities can
%be checked \textit{a posteriori} after each step)
are constructed to ensure \eqref{sqrt8muineq}. As a conclusion,
Theorem \ref{hnconvergencetheorem3rdorder} uses the maximum allowable SSP step size 
 \eqref{ssp-step-size} together with the optimal SSP coefficient $\sspcoeff_n>0$.
\begin{thm}\label{hnconvergencetheorem3rdorder}
Consider the solution of \eqref{ivp} by the third-order formula \eqref{VSSSSPLMMk3}
with $k=4$ or $k=5$.
Let the initial $k-1$ step sizes be
positive and let the subsequent step sizes $h_n$ be chosen according to
\begin{equation}\label{k45hndef}
h_n:=\frac{\sum_{j=1}^{k-1} h_{n-j}}{\left(\sum_{j=1}^{k-1} h_{n-j} \right)+2\mu_n}\cdot\mu_n.
\end{equation}
Assume that \eqref{murestriction}-\eqref{hferestriction}
and the condition
\begin{equation}\label{hjrestriction}
0<h_j\le \varrho\cdot\hFE(u_j)\quad\text{for}\ j=1,\ldots,k-1
\end{equation}
hold with
\begin{equation}\label{particularrhovalues}
(\varrho,\rhoFE)=\begin{cases} 
        \left(\frac{6}{10}, \frac{9}{10}\right) & \text{for } k=4, \\ 
        \left(\frac{57}{100}, \frac{962}{1000}\right) & \text{for } k=5.
\end{cases}
\end{equation}
Then the step-size sequence $h_n$ satisfies
\begin{equation}\label{3ndorderinfsup}
\frac{k-3}{k-1}\,\mu^-\le \liminf_{n\to+\infty} h_n \le \limsup_{n\to+\infty} h_n \le \frac{k-3}{k-1}\,\mu^+.
\end{equation}
%In particular, no subsequence of $h_n$ can converge to $0$. 
As a special case, if $\hFE(u_{n})$ is constant, then 
\[
h_n \to \frac{k-3}{k-1}\,\hFE\quad (n\to +\infty).
\]
\end{thm}
%\begin{rem}
%For $k=4$ and $k=5$, the coefficient $\frac{6-k}{7-k}\cdot (1+\sqrt{2})$ in 
%hypothesis $(S_2)$ is approximately equal to
%$1.609$ and $1.207$, respectively. 
%In view of Remark \ref{remark2}, both assumptions $(S_1)$ and $(S_2)$ are natural. Moreover, 
%the exclusion of the $\Omega_{k-1} > 2(1+\sqrt{2})\approx 4.828$ case is not a severe restriction
%in general, since in the fixed-step-size case $\Omega_{k-1}=k-1\in \{3,4\}$.
%Notice also that in the special case when $\mu_n$ is constant $(=\mu^\pm)$, the assumptions $(S_1)$ and $(S_2)$ reduce to a bound only on the starting values $h_j$. 
%\end{rem}
The proof of this theorem is given in Section \ref{sec:proofs}. From the proof we will see that
it is possible to slightly adjust the simple $(\varrho,\rhoFE)$ values given in 
\eqref{particularrhovalues} based on properties of the problem to be solved; see Figure \ref{doubletrianlges}.

\begin{rem}
Our greedy step-size selection \eqref{k45hndef} for the third-order methods in the above theorem is
optimal in the same sense as it is described in Remark \ref{greedyoptimalremark}.
\end{rem}

\section{Stability and convergence\label{sec:convergence}}
The conditions for a method to have positive SSP coefficient are closely
related to sufficient conditions to ensure stability and convergence.  

Recall that a linear multistep method is said to be zero-stable if it produces a bounded
sequence of values when applied to the initial value problem 
\begin{align} \label{zero-ode} u'(t) & = 0 & u(t_0) & = u_0 & t & \in [t_0,t_0+T] \end{align}
(see, e.g. \cite[Section III.5]{hairer1993}).
The following result shows that variable step-size SSP methods are zero-stable
as long as the step-size restriction for monotonicity is observed.
\begin{thm} \label{thm:stable}
Let $\alpha_{j,n}$ be the coefficients of a variable step-size linear multistep
method \eqref{lmm} and suppose that $\alpha_{j,n}\ge 0$ for all $n,j$.  Then the
method is zero-stable.
\end{thm}
\begin{proof}
Application of \eqref{lmm} to \eqref{zero-ode} yields the recursion
$$ u_n = \sum_{j=0}^{k-1} \alpha_{j,n} u_{n-k+j}.$$
Since the coefficients $\alpha_{j,n}$ are non-negative and sum to one,
each solution value $u_n$ is a convex combination of previous solution values.
Therefore the sequence is bounded.
\end{proof}

\begin{cor}\label{stabilitycorollary}
Let an SSP variable step-size LMM be given and suppose that the step sizes
are chosen so that $\sspcoeff_n>0$ for each $n$.  Then the method is zero-stable.
\end{cor}

In order to prove convergence, we must also bound the local error by bounding
the ratio of successive step sizes and bounding the method coefficients.
%The following lemma shows
%that the SSP step size ensures some nice properties of the method. 
%\begin{lem} \label{lem:bounded-general}
%Let a variable step size method and time step strategy be given such that
%the SSP coefficient of the method is bounded uniformly away from zero: 
%$\sspcoeff_n>\epsilon>0$ with $\epsilon$ independent of $n$.  Then 
%$$0 \le \alpha_{j,n},\beta_{j,n} < \max(1,\epsilon^{-1}) \ \ \ \text{for all } n.$$
%\end{lem}
%\begin{proof}
%    The SSP conditions imply $0 \le \alpha_{j,n}\le 1$ and 
%    $0\le \beta_{j,n} \le \epsilon^{-1} \alpha_{j,n}$.
%\end{proof}
\begin{lem} \label{lem:bounded}
For method \eqref{VSSSSPLMMk2} with $k\ge 3$ or
method \eqref{VSSSSPLMMk3} with $k=4$ or $5$, let the step
sizes be chosen so that $\sspcoeff_n>0$.  Then there exists a constant $\Lambda$
such that 
$$0 \le \alpha_{j,n},\beta_{j,n} < \Lambda \ \ \ \text{for all } n.$$
\end{lem}
\begin{proof}
For any method, $\sspcoeff_n>0$ implies that $\beta_{j,n} \ge 0$ and $0\le \alpha_{j,n} \le 1$.
For the second-order methods, $\sspcoeff_n>0$ implies $\Omega_{k-1,n}>1$, which
implies $\beta_{k-1,n}<2$.  For the third-order methods, $\sspcoeff_n>0$ implies
$\Omega_{k-1,n}>2$, which implies $\beta_{k-1,n}<9/4$ and $\beta_{0,n}<3/4$.
\end{proof}

We recall the following result from \cite[Section III.5, Theorem 5.8]{hairer1993}.
\begin{thm}\label{convergencethm}
Let a variable step-size LMM be applied to an initial value problem with a given
$f$ and on a time interval $[t_0,t_0+T]$.  Let $h=\max_n h_n$ be the largest step in the step
sequence.  Assume that
\begin{enumerate}
    \item the method is stable, of order $p$, and the coefficients $\alpha_{j,n},\beta_{j,n}$
        are bounded uniformly as $h\to 0$;
    \item the starting values satisfy $\|u(t_j)-u_j\| = \order(h_0^p)$ where $h_0$ is a bound on the starting step sizes;
    \item $h_n/h_{n-1}\le\eta$ where $\eta$ is independent of $n$ and $h$.
\end{enumerate}
Then the method is convergent of order $p$; i.e., $\|u(t_n) - u_n\| = \order(h^p)$ for all $t_n\in
[t_0,t_0+T]$.
\end{thm}

The methods described in Section \ref{sec:stepsize} satisfy the conditions of Theorem \ref{convergencethm},
and are thus convergent, as shown in the following theorem.

\begin{thm}\label{ourconvergencethm}
Under the assumptions of Theorem \ref{hnconvergencetheorem}, our methods defined in Section  
\ref{subsectionk2conv} and with $\|u(t_j)-u_j\| = \order(h_0^2)$ for $j\le k$ are convergent of order $2$. Similarly, under the assumptions of Theorem
\ref{hnconvergencetheorem3rdorder}, our methods defined in  
Section \ref{subsectionk3conv} and with $\|u(t_j)-u_j\| = \order(h_0^3)$ for $j\le k$ are convergent of order $3$.
\end{thm}
\begin{proof}
It is sufficient to show that assumptions 1 and 3 of Theorem \ref{convergencethm} are fulfilled.
Our construction in Sections \ref{subsectionk2conv} and 
\ref{subsectionk3conv}  guarantees $\sspcoeff_n>0$.  Therefore Lemma \ref{lem:bounded} applies, so the coefficients $\alpha_{j,n}, \beta_{j,n}$ are uniformly bounded and non-negative. Thus Theorem \ref{thm:stable} applies, so the methods are stable.  Furthermore, the condition $h_n/h_{n-1}\le\eta<+\infty$ 
is implied by \eqref{2ndorderinfsup} or \eqref{3ndorderinfsup}, hence Theorem \ref{convergencethm}
is applicable.
\end{proof}

\section{Numerical examples\label{sec:examples}}
In this section we investigate the performance of the proposed methods by performing numerical tests.
In Section \ref{subsec:convergence}, the accuracy of our methods is verified by a
convergence test on a linear equation with time-varying advection velocity.
In Sections \ref{subsec:burgers}--\ref{subsec:shockbubble} we apply the methods SSPMSV32 and
SSPMSV43 to nonlinear hyperbolic conservation laws in one and two dimensions.  

All code used to generate the numerical results in this work is available at
\url{https://github.com/numerical-mathematics/ssp-lmm-vss_RR}.

\subsection{Efficiency}
Let $N$ denote the number of steps (excluding the starting steps) used to march with a $k$-step variable
step-size method from time $t_0$ to a final time $t_0+T$.
If the same method were used with a fixed step size, the number of steps required would be at least
$N' = (T-\sum_{j=1}^{k-1}h_j)/h_\textup{min}$, where $h_\textup{min}$ is the smallest step size used by
the variable step-size method.
Thus the reduction in computational cost by allowing a variable step size is given by
\begin{align*}
    s := \frac{N}{N'} = \frac{h_\textup{min}}{h_\textup{avg}},
\end{align*}
where $h_\textup{avg}:=(T-\sum_{j=1}^{k-1}h_j)/N$ is the mean step size used by the variable step-size
method.

\subsection{Spatial discretization}
\begin{sloppypar}
In space, we use the wave-propagation semi-discretizations described in \cite{2013_sharpclaw}.
For the second-order temporal schemes, we use a spatial reconstruction based on the 
total-variation-diminishing (TVD) monotonized-central-difference (MC) limiter \cite{VanLeer1977};
for the third-order temporal schemes, we use fifth-order WENO reconstruction.
The only exception is the Woodward--Colella problem, where we use the TVD spatial discretization for both second-
and third-order methods.
The second-order spatial semi-discretization is provably TVD (for scalar, 1D problems) under forward Euler integration
with a CFL number of $\cflFE = 1/2$.
We use this value in the step-size selection algorithm for all methods and all problems.
\end{sloppypar}

\subsection{Time-stepping algorithm}
\begin{sloppypar}
The complete time-stepping algorithm is given in Algorithms \ref{alg2} and \ref{alg3}.
We denote by $\cfl_n$ the CFL number at step $n$.
By default we take the initial step size $h_1=0.1$.
In order to ensure monotonicity of the starting steps, we use the two-stage, second-order SSP
Runge--Kutta method to compute them, with step size
\begin{align*}%\label{starting_methods_stepsize}
	h_n := \gamma \, \sspcoeff_0 \, \hFE(u_{n-1}), \quad n = 1, \dots, k-1,
\end{align*}
where $\sspcoeff_0=1$ is the SSP coefficient of the Runge--Kutta method and $\gamma = 0.9$
is a safety factor (cf. Section~\ref{sec:stepsize}).
\end{sloppypar}

For our third-order methods we can check conditions \eqref{hferestriction} and \eqref{hjrestriction} only
{\em a posteriori}.
%If the sequence of forward Euler step sizes $\{\hFE(u_j)\}_{1\le j \le k}$ is increasing and
%$\gamma \sspcoeff_0 \le \varrho$, then condition \eqref{hjrestriction} is satisfied.
If condition \eqref{hferestriction} or \eqref{hjrestriction} is violated, then the
computed solution is discarded and the current step is repeated with a smaller step size as described in
Algorithm~\ref{alg3}.
These step-size reductions can be repeated if necessary, but in our computations the first reduction was
always sufficient so that the new step was accepted.
For the numerical examples of this section, the step size and CFL number corresponding to the starting
methods are shown on the right of Figures~\ref{stepsize_plots_1D} and~\ref{stepsize_plots_2D}.

\begin{algorithm}[!ht]
\caption{Second-order time-stepping algorithm}
\label{alg2}
\algnewcommand{\block}[2]{\vspace*{#1pt}\State\textbf{#2:}}
\algnewcommand{\mystate}{\State\hspace*{4pt}}
\algrenewcommand\algorithmicfor{\hspace*{10pt}\textbf{for}}
\algrenewtext{EndFor}{\hspace*{10pt}\textbf{end for}}
\algrenewtext{ElsIf}[1]{\hspace*{5pt}\textbf{else if #1 then}}
\algrenewcommand\algorithmicif{\hspace*{5pt}\textbf{if}}
\algrenewcommand\algorithmicelse{\hspace*{5pt}\textbf{else}}
\algrenewcommand\algorithmicend{\hspace*{5pt}\textbf{end}\hspace*{-5pt}}
\begin{algorithmic}
\block{5}{Initialization}
	Choose an initial step size $h_1$ and
		compute the forward Euler step size $\hFE(u_0)$.

\block{5}{Starting procedure}
\For{$n = 1, 2, \dots, k-1$}
	\mystate Compute $u_n$ using the two-stage, second-order SSP Runge--Kutta method with
	\mystate step size $h_n$.
	\mystate Find the CFL number $\cfl_n$ and compute $\hFE(u_n)$.
	\mystate Set $h \gets \gamma \, \sspcoeff_0 \, \hFE(u_n)$.
	\If{$\cfl_n > \sspcoeff_0 \, \cflFE$}
		\State Set $h_n \gets h$ and repeat the step.
	\ElsIf{$n<k-1$}
		\State Set $h_{n+1} \gets h$.
	\Else
		\State Set $h_k$ based on \eqref{general-greedy}.
	\EndIf
\EndFor

\block{5}{Main method}
\For{$n = k, \dots, N$}
    \mystate Compute $u_n$ using the method \eqref{VSSSSPLMMk2} with step size $h_n$.
    \mystate Compute $h_{n+1}$ from \eqref{general-greedy}.
\EndFor
\end{algorithmic}
\end{algorithm}

\begin{algorithm}[!ht]
\caption{Third-order time-stepping algorithm}
\label{alg3}
\algnewcommand{\block}[2]{\vspace*{#1pt}\State\textbf{#2:}}
\algnewcommand{\mystate}{\State\hspace*{4pt}}
\algrenewcommand\algorithmicfor{\hspace*{10pt}\textbf{for}}
\algrenewtext{EndFor}{\hspace*{10pt}\textbf{end for}}
\algrenewtext{ElsIf}[1]{\hspace*{5pt}\textbf{else if #1 then}}
\algrenewcommand\algorithmicif{\hspace*{5pt}\textbf{if}}
\algrenewcommand\algorithmicelse{\hspace*{5pt}\textbf{else}}
\algrenewcommand\algorithmicend{\hspace*{5pt}\textbf{end}\hspace*{-5pt}}
\begin{algorithmic}
\block{5}{Initialization}
    Choose an initial step size $h_1$ and compute the forward Euler step size $\hFE(u_0)$.

\block{5}{Starting procedure}
\For{$n = 1, 2, \dots, k-1$}
    \mystate Compute $u_n$ using the two-stage, second-order SSP Runge--Kutta method with
	\mystate step size $h_n$.
    \mystate Find the CFL number $\cfl_n$ and compute $\hFE(u_n)$.
    \If{condition \eqref{hferestriction} or \eqref{hjrestriction} does not hold}
        set
        \begin{align*}
            h_n \gets \begin{cases}
                h_n/2 & \text{ if at least \eqref{hferestriction} is violated;} \\
                    \gamma \, \sspcoeff_0 \, \varrho \, \hFE(u_n)  &
                        \text{ if \eqref{hjrestriction} is violated,}
            \end{cases}
        \end{align*}
        \State and repeat the step.
    \Else 
        \State Set $h \gets \gamma \, \sspcoeff_0 \, \hFE(u_n)$.
  \EndIf
	\If{$\cfl_n > \sspcoeff_0 \, \cflFE$}
		\State Set $h_n \gets h$ and repeat the step.
	\ElsIf{$n<k-1$}
		\State Set $h_{n+1} \gets h$.
	\Else
		\State Set $h_k$ based on \eqref{k45hndef}.
	\EndIf
\EndFor

\block{5}{Main method}
\For{$n = k, \dots, N$}
    \mystate Compute $u_n$ using the method \eqref{VSSSSPLMMk3} with step size $h_n$.
    \If{condition \eqref{hferestriction} is violated}
        \State Set $h_n \gets h_n/2$ and repeat step.
    \Else
	\State Compute $h_{n+1}$ from \eqref{k45hndef}.
    \EndIf
\EndFor
\end{algorithmic}
\end{algorithm}
\begin{rem}
	There is no need to check the CFL condition $\cfl_n \le \sspcoeff_0 \, \cflFE$ for
	$n \ge k$ (i.e. when the LMM method is used) since it is automatically satisfied by the step size
	selection \eqref{ssp-step-size}.
\end{rem}

Condition \eqref{hferestriction} is violated only when the maximum wave speed changes
dramatically between consecutive steps, which may suggest insufficient temporal
resolution of the problem.
We have found that, for the problems considered herein, omitting the enforcement of
condition \eqref{hjrestriction} never seems to change the computed solution in a
significant way.  Since these two conditions were introduced only as technical
assumptions for some of our theoretical results, they could perhaps be omitted
in a practical implementation.

\subsection{Convergence test}\label{subsec:convergence}
We consider the linear advection problem
\begin{align} \label{advection}
    u_t + \left( 2 + \frac{3}{2}\sin(2 \pi t) \right) u_x = 0, \\
    u(x,0) = \sin(2\pi x), \quad x \in [0,1], \nonumber
\end{align} 
with periodic boundary conditions.
We use a spatial step size $\Delta x = 2^{-d}$, $d = 6,\dots,11$, and compute the solution at a final time 
$t = 5$ (i.e., after 10 cycles).
Table~\ref{convergencetest} shows the $\text{L}_1$-norm of the error at the final time.
The solution is computed by using second-order methods with $k = 3, 4$ steps and third-order methods
with $k = 4, 5$ steps.
All methods attain the expected order.

\begin{table}
    \centering
    \small
	\begin{tabular*}{\textwidth}{@{\extracolsep{\fill}} c|cc|cc|cc|cc}
		\toprule
		\multicolumn{1}{c}{$N$} & \multicolumn{2}{c}{SSPMSV32} & \multicolumn{2}{c}{SSPMSV42} & \multicolumn{2}{c}{SSPMSV43} & \multicolumn{2}{c}{SSPMSV53} \\
		\midrule
		$128$ & $1.50 \times 10^{-2}$ & & $1.83 \times 10^{-2}$ & & $9.20 \times 10^{-6}$ & & $6.08 \times 10^{-5}$ & \\
		$256$ & $4.30 \times 10^{-3}$ & $1.80$ & $5.34 \times 10^{-3}$ & $1.78$ & $1.30 \times 10^{-6}$ & $2.82$ & $8.10 \times 10^{-6}$ & $2.91$ \\
		$512$ & $1.15 \times 10^{-3}$ & $1.90$ & $1.44 \times 10^{-3}$ & $1.89$ & $1.68\times 10^{-7}$ & $2.95$ & $1.04 \times 10^{-6}$ & $2.96$ \\
		$1024$ & $3.01 \times 10^{-4}$ & $1.93$ & $3.81 \times 10^{-4}$ & $1.92$ & $2.13 \times 10^{-8}$ & $2.98$ & $1.32 \times 10^{-7}$ & $2.98$ \\
		$2048$ & $7.74 \times 10^{-5}$ & $1.96$ & $9.84 \times 10^{-5}$ & $1.95$ & $2.67 \times 10^{-9}$ & $2.99$ & $1.66 \times 10^{-8}$ & $2.99$\\
		\bottomrule
	\end{tabular*}
	\caption{$\text{L}_1$-norms of the error at final time $t = 5$ for the variable-coefficient
	advection problem \eqref{advection} for second- and third-order LMMs.
	For each method, the second column denotes the convergence order.
	The number of spatial points is indicated by $N$.}
    \label{convergencetest}
\end{table}

%\begin{figure}
%	\begin{center}
%	\includegraphics[width=0.7\textwidth]{figures/conv_advection_1d}
%	\caption{$L_1$-norm errors of linear advection convergence study for second- and third-order LMMs.
%	For each method the second column denotes the convergence order.}
%	\label{conv_advection_1d}
%	\end{center}
%\end{figure}

\subsection{Burgers' equation}\label{subsec:burgers}
We consider the inviscid Burgers' initial value problem
\begin{gather*}
	u_t + u u_x = 0, \\
	u(x,0) = \frac{1}{2} + \sin(2\pi x), \quad x \in [0,1],
\end{gather*}
with periodic boundary conditions and $256$ grid cells.
Figure~\ref{stepsize_burgers_1d} shows the evolution of the step size, up to a final time $t = 0.8$.
The step size is constant until the shock forms, and then increases.
The dashed curves in Figure~\ref{stepsize_burgers_1d} indicate how the CFL number $\cfl_n$ varies with
time for each method.
Since $\cfl_n = \sspcoeff_n \cflFE$ and $\cflFE = 1/2$, the CFL number after the first few steps is $1/4$ and
$1/6$, for the second- and third-order LMMs, respectively.
Note that these are the theoretical maximum values for which the solution remains TVD.
For the third-order method coupled with WENO discretization the TVD-norm of the solution does not increase more than $10^{-4}$ over a single time step.
The sudden decrease of the step size after the first few steps of the simulation is due to the switch
from the starting (Runge--Kutta) method to the linear multistep method.
%The SSP coefficient of the starting method is equal to $1$, thus the allowed CFL number for the starting
%values is about twice as large relative to the CFL number of the second-order LMM and three times as large
%relative to the CFL number of the third-order method.
For both methods we have $s \approx 0.88$.

\subsection{Woodward--Colella blast-wave problem}\label{subsec:WCblast}
Next we consider the one-dimensional Euler equations for inviscid, compressible flow:
\begin{align*}
	 \rho_t + (\rho u)_x  & = 0 \\
	(\rho u)_t + (\rho u^2 + p)_x & = 0 \\
	E_t + (u (E + p) )_x & = 0.
\end{align*}
Here $\rho$, $u$, $E$ and $p$ denote the density,  velocity, total energy and pressure, respectively.
The fluid is an ideal gas, thus the pressure is given by $p=\rho(\gamma - 1)e$, where $e$ is internal energy
and $\gamma = 1.4$.
The domain is the unit interval with $\Delta x = 1/512$, and we solve the Woodward--Colella blast-wave
problem \cite{WoodwardCollelablast} with initial conditions
\begin{align*}
\rho(x,0) = 1, \quad v(x,0) = 0, \quad \text{ and } \quad p(x,0) = \begin{cases} 
																					1000 & \mbox{if } 0 \le x < 0.1 \\
																					0.01 & \mbox{if } 0.1 \le x < 0.9 \\
																					100 & \mbox{if } 0.9 \le x \le 1
																				  \end{cases}.
\end{align*}
The initial conditions consist of two discontinuities. 
Reflecting boundary conditions are used and initial conditions lead to strong shock waves,
and rarefactions that eventually interact with each other.
The second-order TVD semi-discretization is used with both the second- and third-order methods and the
solution is computed at time $t = 0.04$, after the two shock waves have collided.

Figure~\ref{stepsize_woodward_colella_1d} shows how the step size evolves over time.
The maximum wave speed decreases as the shock waves approach each other, then increases when they
collide, and finally decreases when the shocks move apart.
The step size exhibits exactly the opposite behavior since it is inversely proportional to the maximum wave
speed.
As before, the CFL number $\cfl_n$ remains close to $1/4$ and $1/6$ for the second- and third-order LMMs,
respectively.
Note that these are the theoretical maximum values for which the characteristic variables remain TVD.
For this problem $s \approx 0.76$.

\begin{figure*}[!ht]
	\centering
	\subfloat[Burgers' equation]{\label{stepsize_burgers_1d}
	\defaultgraphic % insert burgers' plot as the default size figure
	\hspace*{5pt}
	\includegraphics[height=\graphicheight]{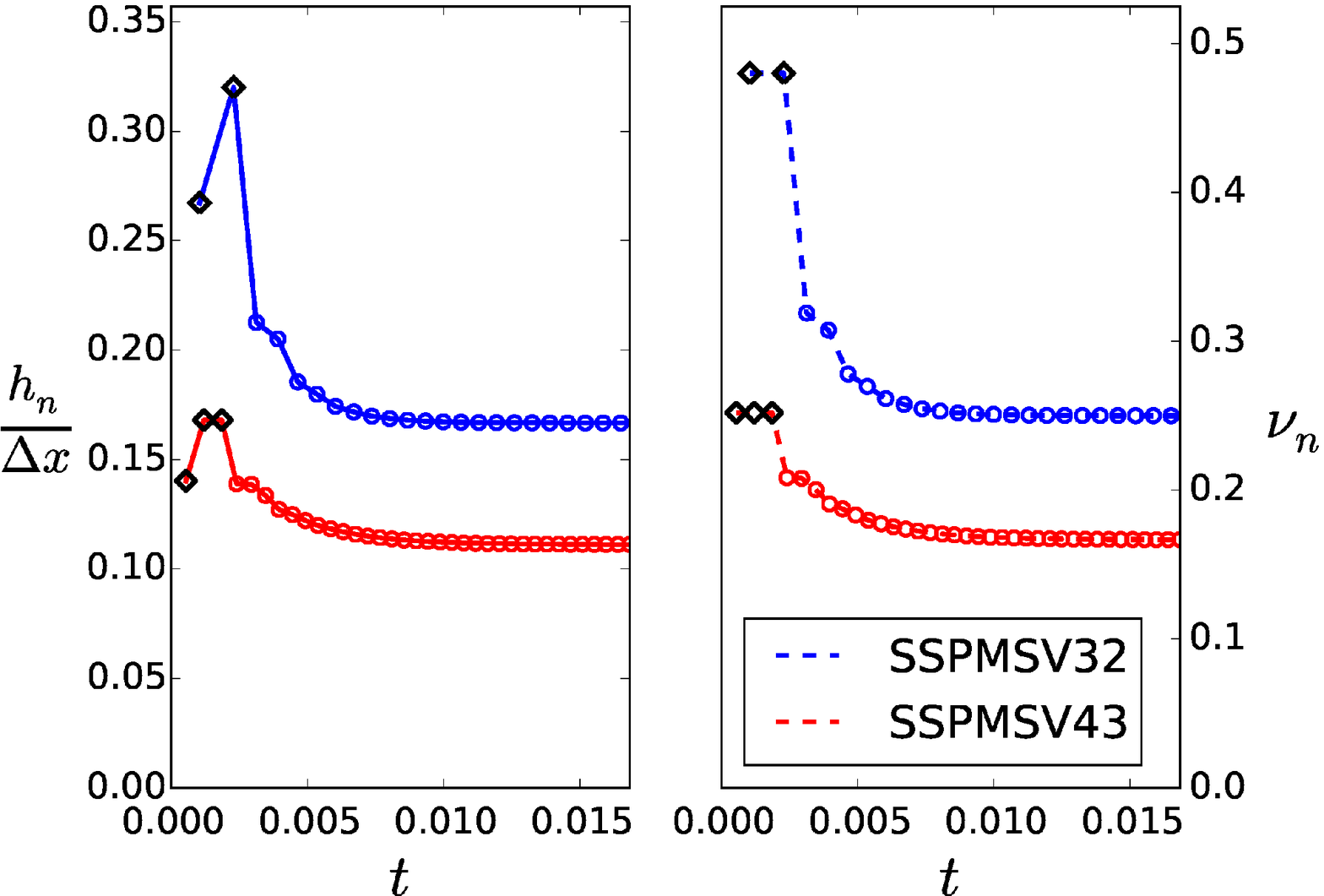}}
	\newline
	\subfloat[Woodward--Colella blast-wave problem]{\label{stepsize_woodward_colella_1d}
	\includegraphics[height=\graphicheight]{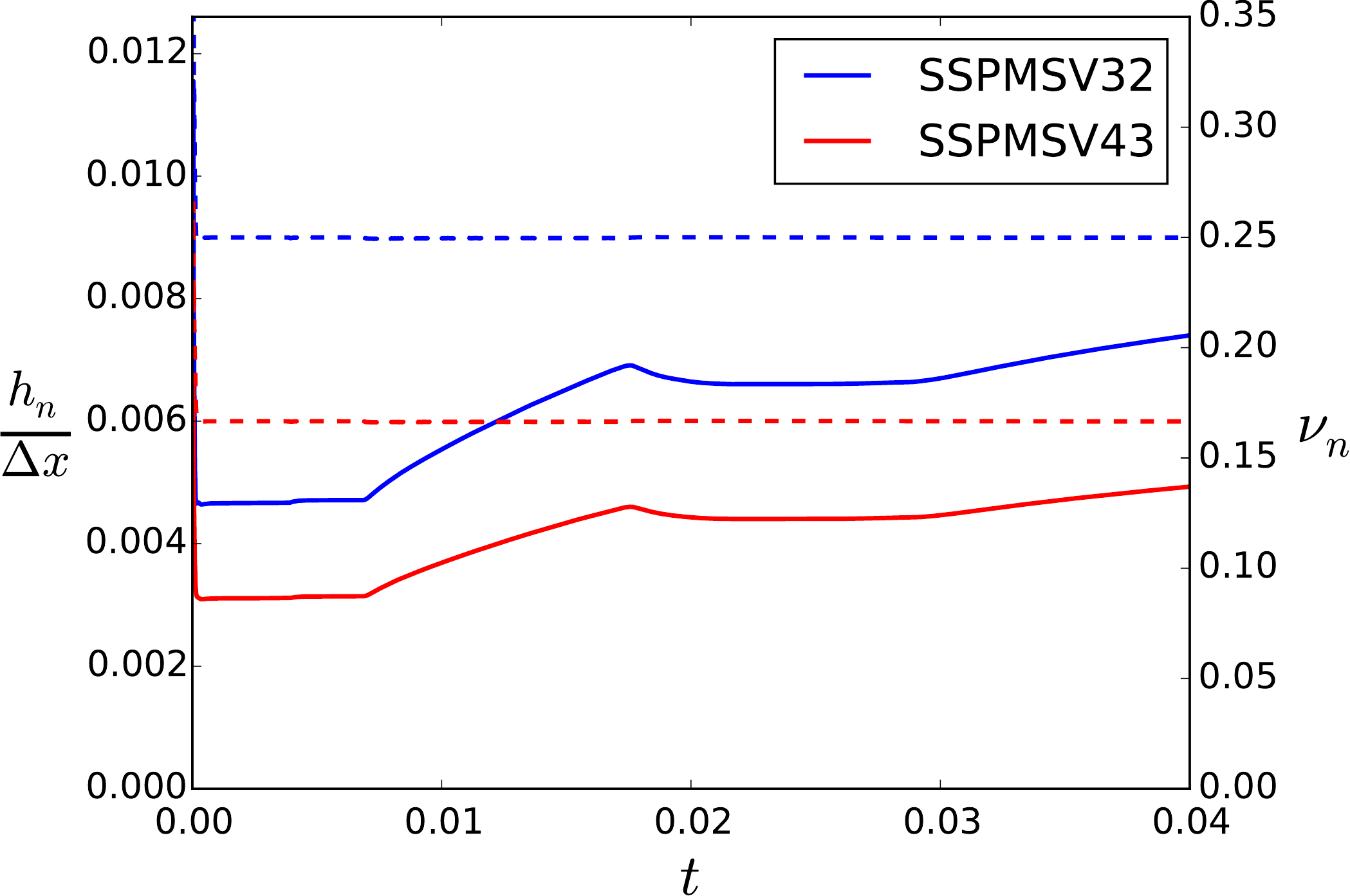}
	\hspace*{-1pt}
	\includegraphics[height=\graphicheight]{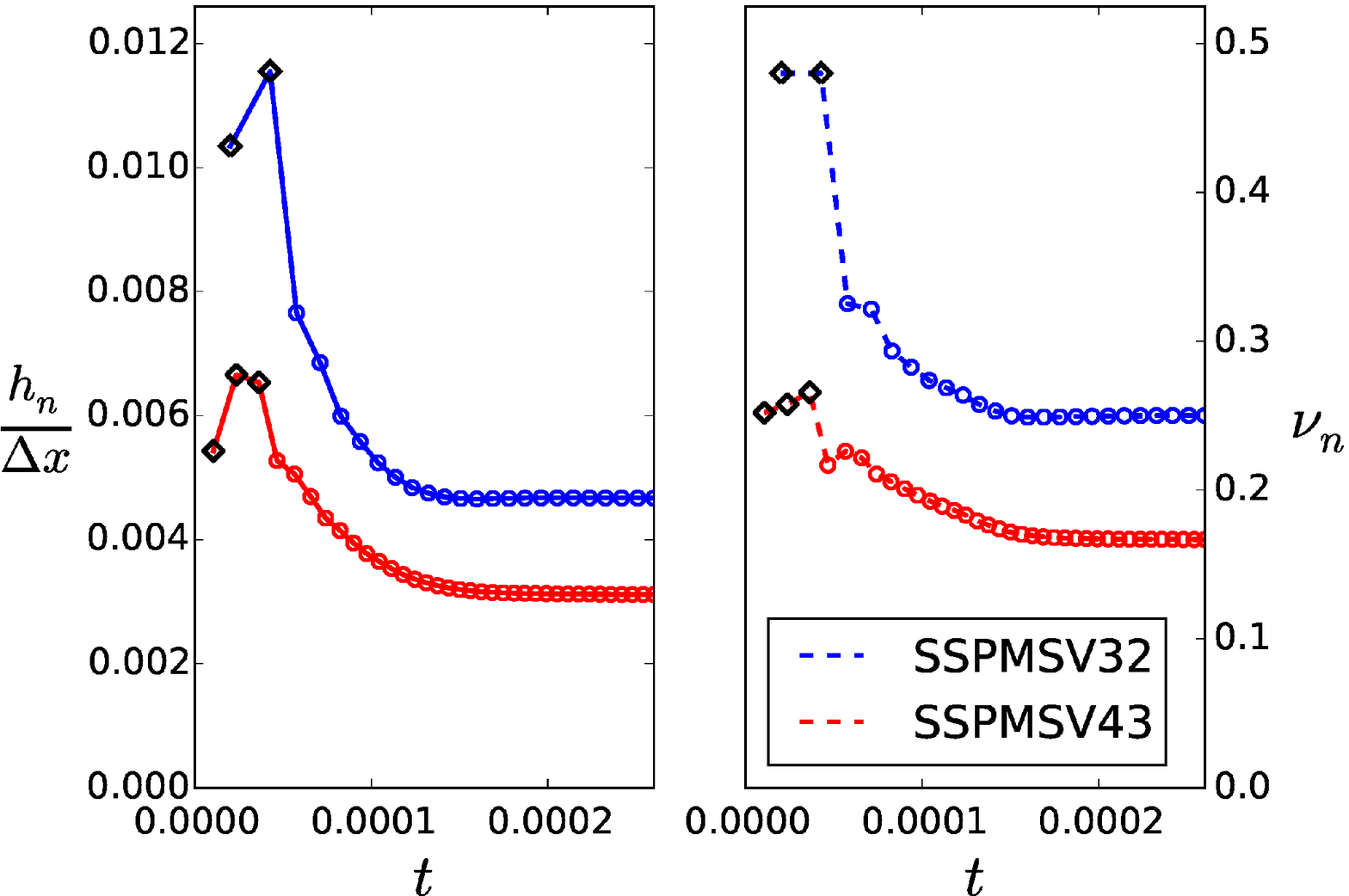}}
	\caption{Evolution of step size for one-dimensional Burgers' and Euler equations with
	second- and third-order linear multistep methods (solid lines).
    The dashed curves indicate the CFL number for each method.
    Figures on the right show a close-up view of the step size and CFL number at early times.
    The black marks indicate the starting Runge--Kutta method's step sizes.}
    \label{stepsize_plots_1D}
\end{figure*}

\subsection{Two-dimensional shallow-water flow}\label{subsec:SW2D}
\begin{sloppypar}
Consider the two-dimensional shallow-water equations
\begin{align*}
	 h_t + (hu)_x + (h v)_y & = 0 \\
	(hu)_t + (h u^2 + \frac{1}{2}gh^2)_x + (huv)_y & =0 \\
	(hv)_t + (huv)_x + (hv^2 + \frac{1}{2}gh^2)_y & = 0,
\end{align*}
\end{sloppypar}\noindent
where $h$ is the depth, $(u, v)$ the velocity vector and $hu$, $hv$ the momenta in each direction.
The gravitational constant $g$ is equal to unity.
The domain is a square $[-2.5, 2.5] \times [-2.5, 2.5]$ with $250$ grid cells in each direction.
The initial condition consists of an uplifted, inward-flowing 
%annular region 
cylindrically symmetric perturbation
given by
\begin{align*}
h(x,y,0) = 1+ g(x,y), \quad
u(x,y,0) = -x g(x,y) \quad \text{ and } \quad
v(x,y,0) = -y g(x,y),
\end{align*}
where $g(x,y) := e^{-10\left(\sqrt{x^2+y^2}-1\right)^2}$.
We apply reflecting boundary conditions at the top and right boundary, whereas the bottom and left
boundaries are outflow.
The wave is initially directed towards the center of the domain, and a large peak forms as it converges
there, before subsequently expanding outward.
We compute the solution up to time $t = 2.5$, after the solution has been reflected from the top and right
boundaries.
Figure~\ref{stepsize_radial_shallow_water_2d} shows how the step size varies in time.
It decreases as the initial profile propagates towards the center of the domain since the maximum wave
speed increases.
As the solution expands outwards, higher step sizes are allowed. 
Note that at time $t \approx 2$ the wave hits the boundaries, so a small decrease in the allowed step size is
observed.
Again, the dashed curves indicate the variation of the CFL number.
Now we have the ratio $s \approx 0.62$ for both second- and third-order methods.

\subsection{Shock-bubble interaction}\label{subsec:shockbubble}
Finally, we consider the Euler equations of compressible fluid dynamics in cylindrical coordinates.
Originally, the problem is three-dimensional but can been reduced to two dimensions by using cylindrical
symmetry.  The resulting system of equations is

\begin{align*}
	 \rho_t + (\rho u)_z + (\rho v)_r & = - \frac{\rho v}{r} \\
	(\rho u)_t + (\rho u^2 + p)_z + (\rho uv)_r & = -\frac{\rho u v}{r} \\
	(\rho v)_t + (\rho uv)_z + (\rho v^2 + p)_r & = - \frac{\rho v^2}{r} \\
	(\rho E)_t + (u (\rho E + p) )_z + (v (\rho E + p))_r & = - \frac{(\rho E + p) v}{r}.
\end{align*}
Here, $\rho$ is the density, $p$ is the pressure, $E$ is the total energy, while $u$ and $v$ are the $z$- and
$r$-components of the velocity.
The $z$- and $r$-axis are parallel and perpendicular, respectively, to the axis of symmetry.
The problem involves a planar shock wave traveling in the $z$-direction that impacts a spherical low-density
bubble.
The initial conditions are taken from \cite{Ketcheson2012pyclaw}.
We consider a cylindrical domain $[0, 2] \times[0, 0.5]$ using a $640 \times 160$ grid and impose reflecting
boundary conditions at the bottom of the domain and outflow conditions at the top and right boundaries.

Since the fluid is a polytropic ideal gas, the internal energy $e$ is proportional to the temperature.
Initially the bubble is at rest, hence there is no difference in pressure between the bubble and the 
surrounding air.
Therefore, the temperature inside the bubble is high resulting in large sound speed 
$c = \sqrt{\gamma(\gamma-1)e}$, where $\gamma = 1.4$.
Consequently, this results in reduced step sizes right after the shock wave hits the bubble at 
$t \approx 0.05$, as shown in Figure~\ref{stepsize_shock_bubble_interaction_2d}.
The efficiency ratio $s$ is about $0.82$. 

\begin{figure*}[!ht]
	\centering
	\subfloat[Shallow-water equations]{\label{stepsize_radial_shallow_water_2d}
	\includegraphics[height=\graphicheight]{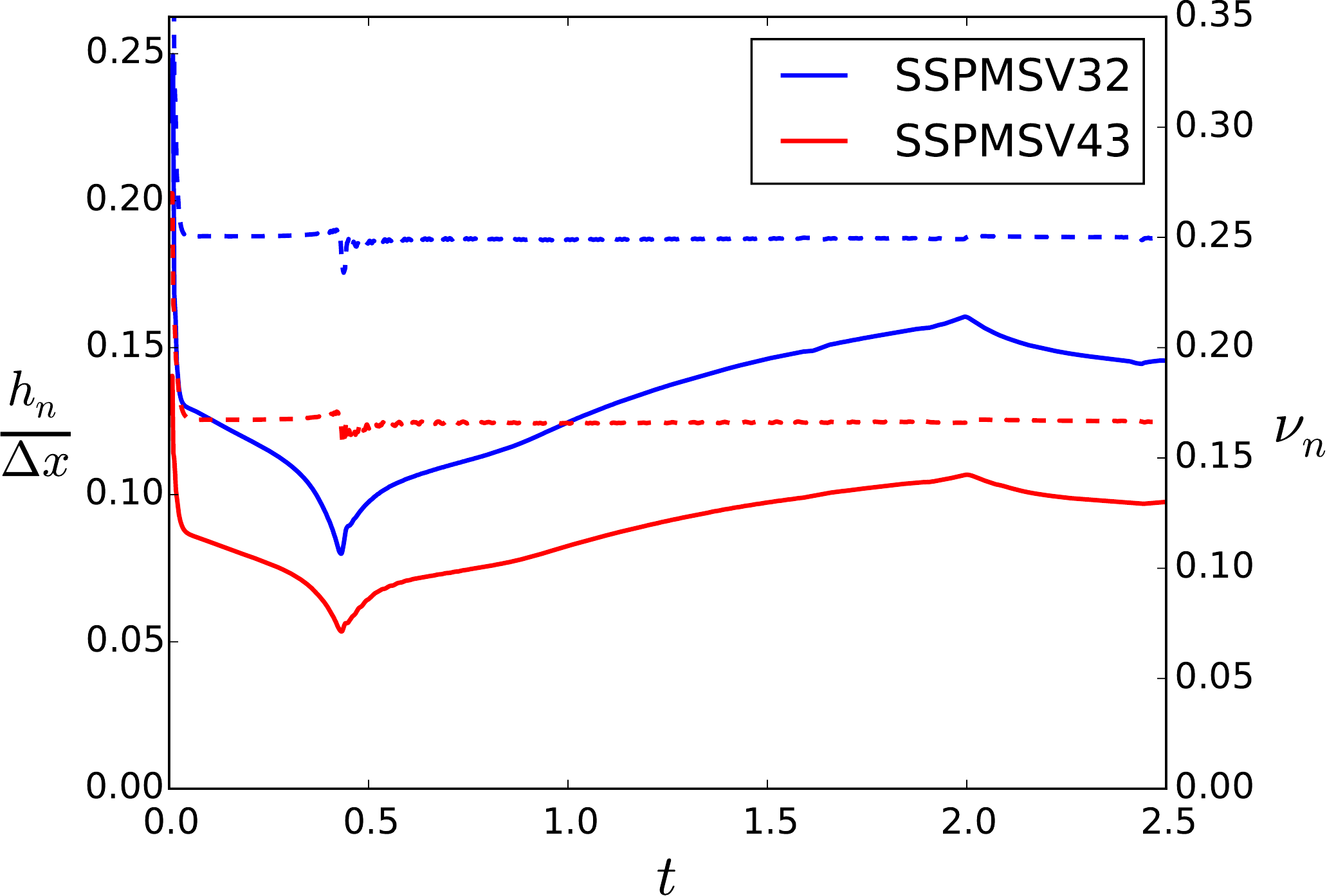}
	\hspace*{5pt}
	\includegraphics[height=\graphicheight]{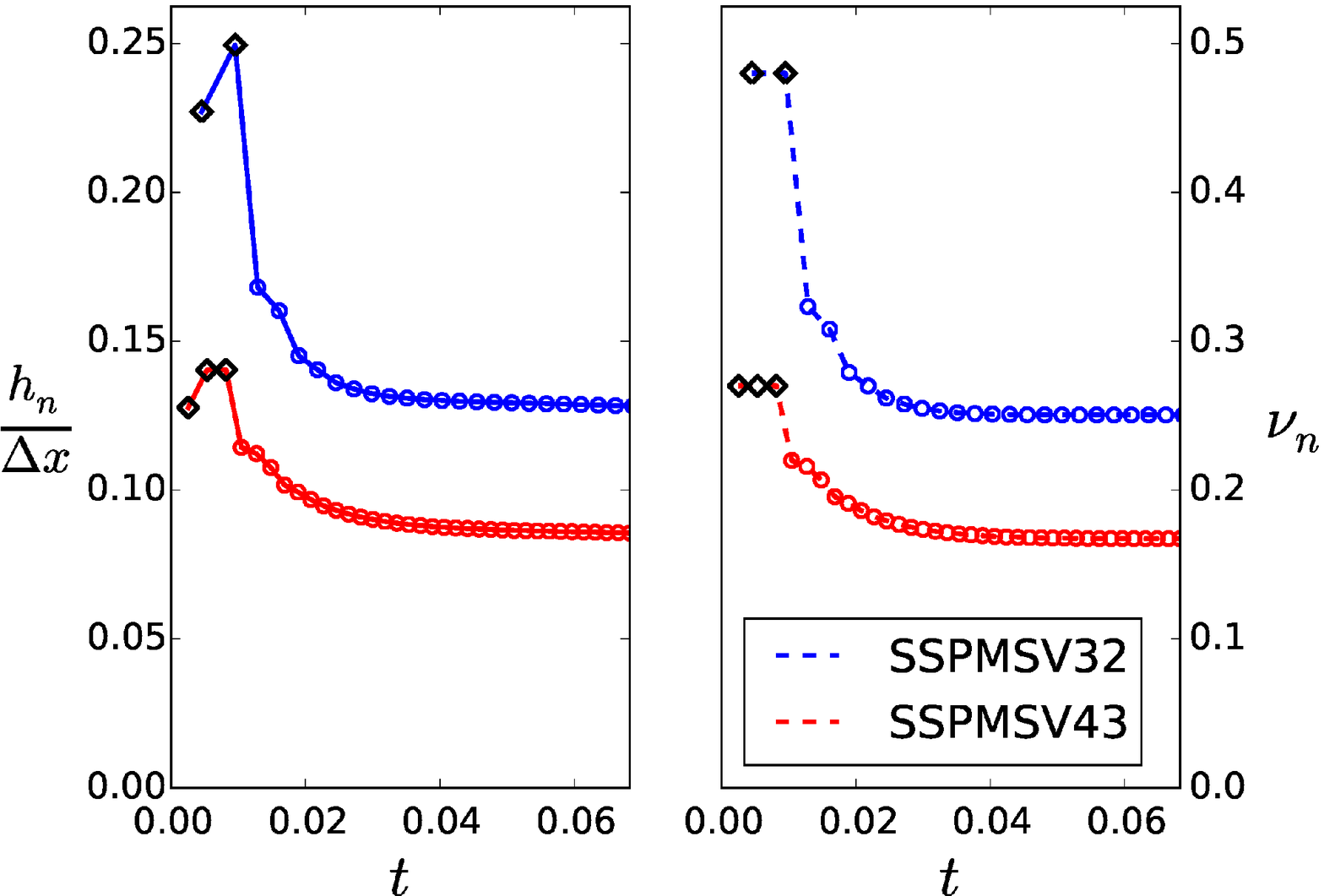}}
	\newline
	\subfloat[Shock-bubble interaction problem]{\label{stepsize_shock_bubble_interaction_2d}
	\includegraphics[height=\graphicheight]{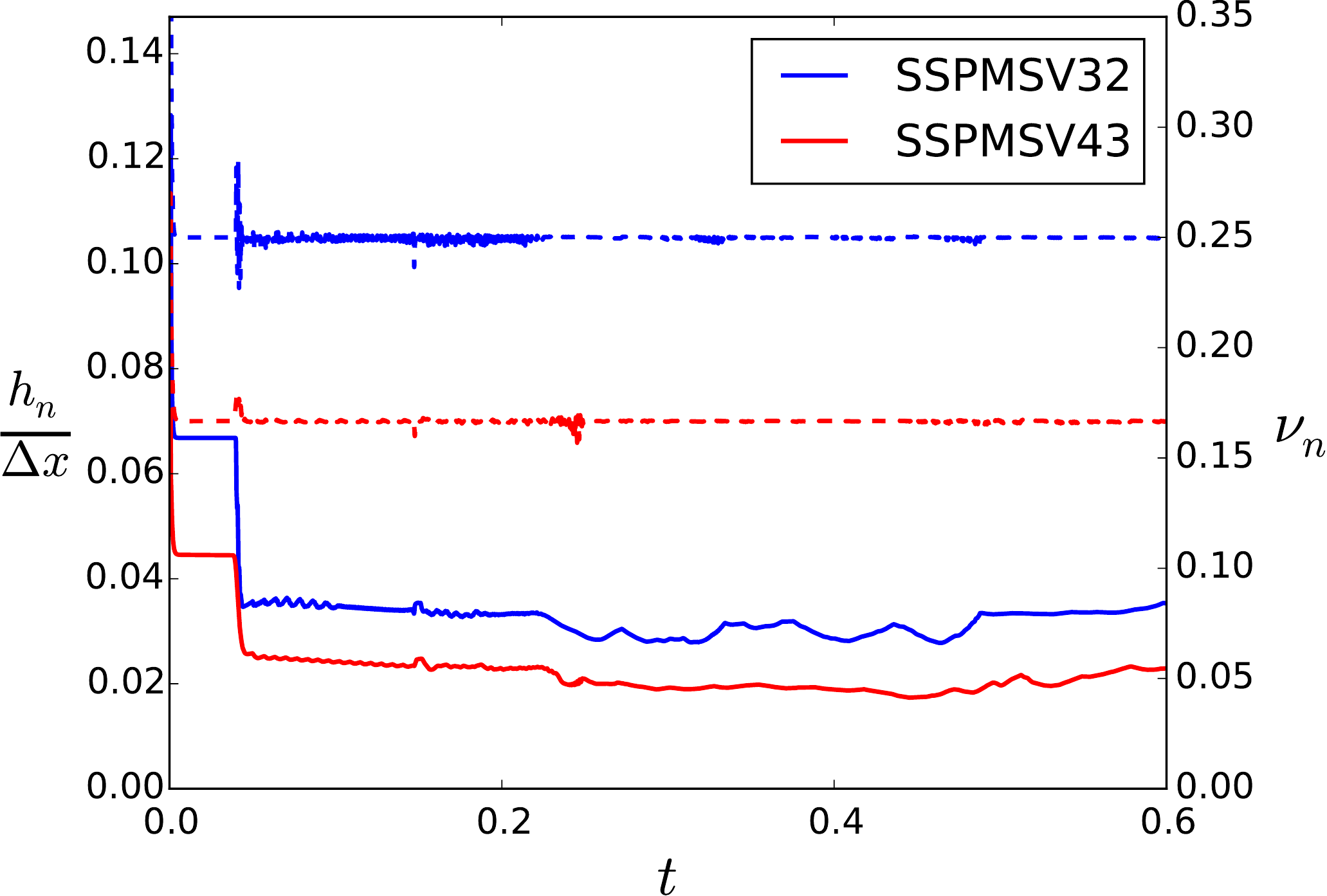}
	\hspace*{5pt}
	\includegraphics[height=\graphicheight]{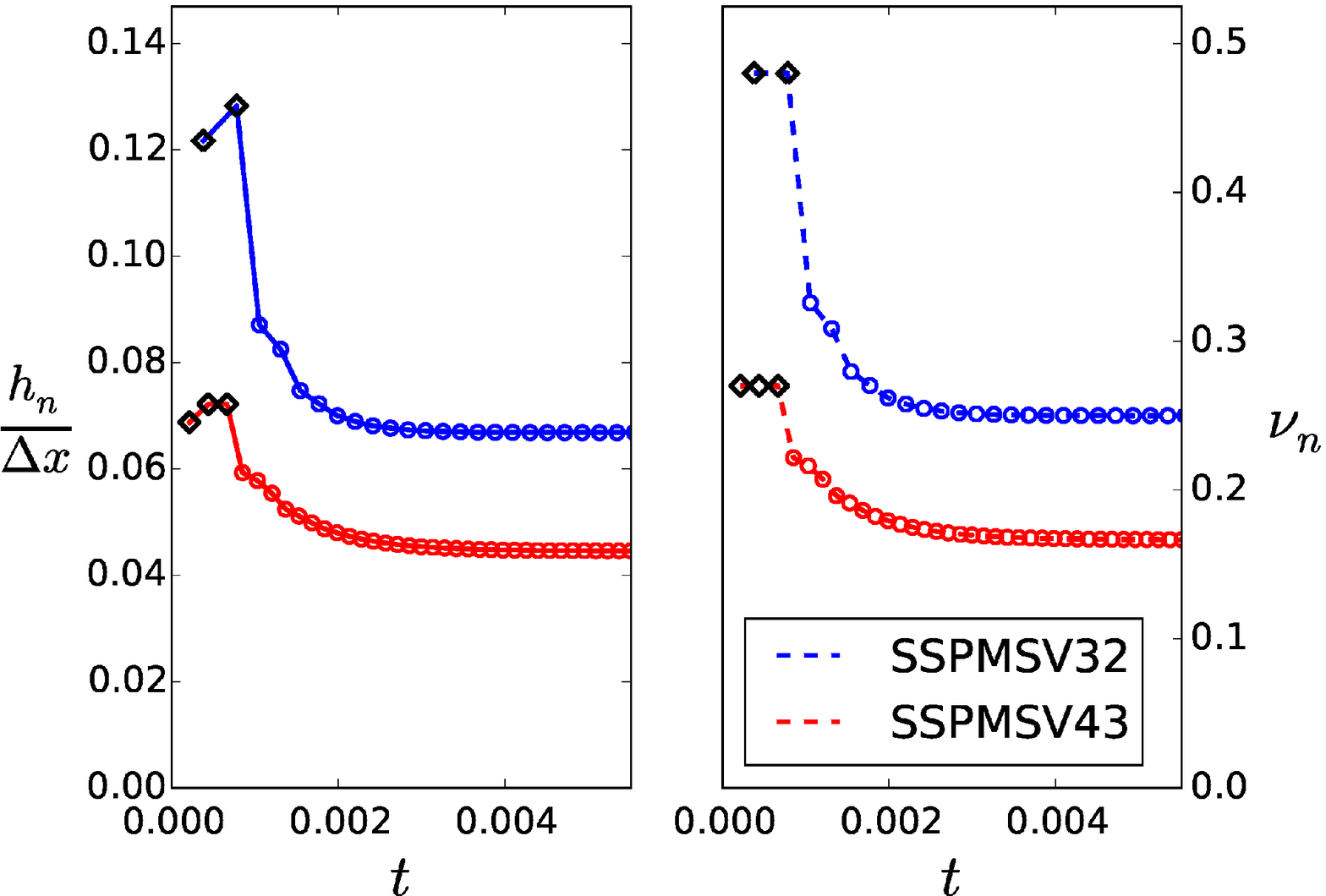}}
	\caption{Evolution of step size for two-dimensional shallow-water and Euler equations with
	second- and third-order linear multistep methods (solid lines).
	The dashed curves indicate the CFL number for each method.
	Figures on the right show a close-up view of the step size and CFL number at early times.
    The black marks indicate the starting Runge--Kutta method's step sizes.}
    \label{stepsize_plots_2D}
\end{figure*}

\section{Conclusions}
The methods presented in this paper are the first adaptive multistep methods
that provably preserve any convex monotonicity property satisfied under 
explicit Euler integration.  We have provided a step-size selection 
algorithm---yielding an optimal step-size sequence $h_n$---that strictly enforces this property while ensuring that the step size is
bounded away from zero.  The methods are proved to converge at the expected rate
for any ODE system satisfying the forward Euler condition.

As suggested by Theorems~\ref{hnconvergencetheorem} and
~\ref{hnconvergencetheorem3rdorder}, for all tests the CFL number remains close to $\frac{k-p}{k-1}\cflFE$,
where $k$ and $p$ are the steps and order of the method, respectively.
The numerical results verify that the step size is approximately inversely proportional to the maximum wave
speed and the proposed step-size strategy successfully chooses step sizes closely related to the maximum
allowed CFL number.
In practice, we expect that conditions \eqref{hferestriction} and
\eqref{hjrestriction}, which were introduced only as technical assumptions for
some of our theoretical results, could usually be omitted without impacting the
solution.

The methods presented herein are of orders two and three.
We have proved the existence of methods of arbitrary order. However, the optimal methods
for orders at least 3 seem to have a complicated structure.
It would be useful to develop adaptive (and possibly suboptimal) multistep SSP methods of order higher
than three having a simple structure.
%There is no theoretical reason why this cannot be done, but the approach
%used in the present work leads to systems of equations and inequalities
%with no straightforward symbolic solution.

\section{The proofs of the theorems in Section \ref{sec:exandupper}\label{sec:proofs0}}
We present three lemmas, Lemma \ref{lem:dual} below, Lemma \ref{lem:farkas} in Section \ref{sec:proof0Thm2proof}
and Lemma \ref{lem:dualoptimal} in Section \ref{sec:proof0Thm3proof}, that are used in the 
proofs of Theorems \ref{thm:upperbound}-\ref{sec3Thm5} in Sections \ref{sec:proof0Thm2proof}-\ref{sec:proof0Thm5proof}.
Lemmas \ref{lem:dual}-\ref{lem:dualoptimal} are straightforward generalizations of the corresponding
results of \cite{LMMpaper} 
for variable step-size formulae, hence their proofs are omitted here.
\begin{lem}
  \label{lem:dual}
  Let $r \ge 0$ be arbitrary and let $p$ and $n\ge k$ be arbitrary positive integers.
  Suppose that some time-step ratios $\omega_{j,n}>0$ are given.
  Then the following two statements are equivalent.
  \begin{enumerate}
    \item[(i)]
      For all formulae with $k$ steps and order of accuracy $p$ we have
      $\sspcoeff_n(\omega, \delta, \beta) \leq r$.
    \item[(ii)]
      There exists a non-zero real univariate polynomial $q_n$ of degree at most $p$, that satisfies the conditions
      \begin{subequations} \label{dualconditions}
        \begin{align}
          q_n(\Omega_{j,n}) &\geq 0 & \text{ for all } 0 \leq j \leq k-1, \label{dualconditions:a}\\
          q_n'(\Omega_{j,n})+rq_n(\Omega_{j,n}) &\geq 0 & \text{ for all } 0 \leq j \leq k-1, \label{dualconditions:b}\\
          q_n(\Omega_{k,n}) &= 0. \label{dualconditions:c}
        \end{align}
      \end{subequations}
  \end{enumerate}
  Furthermore, if statement $(i)$ holds, then the polynomial $q_n$ can be
chosen to satisfy the following conditions.
  \begin{enumerate}[leftmargin=35pt]
    \item The degree of $q_n$ is exactly $p$.
    \item The real parts of all roots of $q_n$ except for $\Omega_{k,n}$ lie in the
      interval $[\Omega_{0,n},\Omega_{k-1,n}]$.
    \item The set of all real roots of $q_n$ is a subset of $\{\Omega_{0,n},\Omega_{1,n},\ldots,\Omega_{k,n}\}$.
    \item For even $p$, $\Omega_{0,n}$ is a root of $q_n$ with odd multiplicity.
      For odd $p$, if $\Omega_{0,n}$ is a root of $q_n$, then its multiplicity is even.
    \item The multiplicity of the root $\Omega_{k,n}$ is one. For $1\leq j \leq k-1$, 
if $\Omega_{j,n}$ is a root of $q_n$,
      then its multiplicity is 2.
    \item The polynomial $q_n$ is non-negative on the interval $[\Omega_{0,n},\Omega_{k,n}]$.
  \end{enumerate}
\end{lem}

The following observation will be useful in the proofs. 
  If $q_n(\Omega_{j,n}) = 0$, then the inequality
  \eqref{dualconditions:b} for this index $j$ simplifies to
   $q_n'(\Omega_{j,n}) \geq 0$.
  Otherwise, if $q_n(\Omega_{j,n}) \ne 0$, then the inequality
  \eqref{dualconditions:b} for this index $j$ can be written by using
the logarithmic derivative of $q_n$ as
  \begin{equation*}
    \frac{q_n'(\Omega_{j,n})}{q_n(\Omega_{j,n})} = \sum\limits_{\ell=1}^{p}\frac{1}{\Omega_{j,n}-\lambda_{\ell,n}} \geq -r,
  \end{equation*}
where the $\lambda_{\ell,n}$ numbers are the complex (including real) roots of $q_n$, 
and $q_n$ has been chosen such that its degree is exactly $p$. 

\subsection{The proof of Theorem \ref{thm:upperbound}}\label{sec:proof0Thm2proof}
The following lemma will be used in the second step of the proof of Theorem~\ref{thm:upperbound}.
\begin{lem}\label{lem:farkas}
  Let $r > 0$ be arbitrary and let $p$ and $n\ge k$ be arbitrary positive integers.
  Suppose that some time-step ratios $\omega_{j,n}>0$ are given.
  Then exactly one of the following statements is true.
  \begin{enumerate}
    \item[(i)] There is a formula with $k$ steps, order of accuracy $p$ and
      $\sspcoeff_n(\omega, \delta, \beta) \ge r$.
    \item[(ii)] There is a real univariate polynomial $q_n$ of degree at most $p$,
      that satisfies the conditions \eqref{dualconditions:a}, \eqref{dualconditions:b}
      and
      \begin{equation}\label{dualconditions:c:farkas}
        q_n(\Omega_{k,n}) < 0.
      \end{equation}
  \end{enumerate}
\end{lem}
\begin{proof}[The proof of Theorem \ref{thm:upperbound}] 

\textbf{Step 1.}
To  prove the bound \eqref{boundonComegadeltabeta}, 
 we first consider the $k\ge 2$ case, and let $q_n$ denote the polynomial
\begin{equation}\label{Thm2qndef} 
q_n(x) := x^{p-1}(\Omega_{k,n}-x).
\end{equation}
%\begin{itemize}
\noindent \quad $\bullet$ If $\Omega_{k,n}>p$, this $q_n$ satisfies conditions \eqref{dualconditions} with  $r:=
\left({\Omega_{k,n}-p}\right)/\left({\Omega_{k,n}-1}\right)$, since
for each $1 \leq j \leq k-1$ we have
\[
%  \begin{align*}
    \frac{q_n'(\Omega_{j,n})}{q_n(\Omega_{j,n})} = \frac{p-1}{\Omega_{j,n}}-\frac{1}{\Omega_{k,n}-\Omega_{j,n}}
    \geq \frac{p-1}{\Omega_{k-1,n}}-\frac{1}{\Omega_{k,n}-\Omega_{k-1,n}} = -\frac{\Omega_{k,n}-p}{\Omega_{k,n}-1},
\]
%    & ,
% \end{align*}
and $q_n'(\Omega_{0,n})\equiv q_n'(0)\ge 0$.

\noindent \quad $\bullet$ If $\Omega_{k,n}\le p$, the conditions \eqref{dualconditions} hold with $r:=0$, because for $x\in [\Omega_{0,n},\Omega_{k-1,n}]\equiv [0,\Omega_{k,n}-1]$ we have
\[
q_n'(x)\ge x^{p-2} \left((p-1) \Omega_{k,n}-p \left(\Omega_{k,n}-1\right)\right)=
x^{p-2} \left(p-\Omega_{k,n}\right)\ge 0.
\]
For $k=1$, the same $q_n$ as in \eqref{Thm2qndef} satisfies
\eqref{dualconditions} with $r:=0$ and $\Omega_{k,n}=1$. 
Then, in each case, \eqref{boundonComegadeltabeta} is an immediate consequence of Lemma~\ref{lem:dual}.

\textbf{Step 2.} Suppose now 
\begin{equation}\label{Thm2Step2assumption}
\exists \text{ a } k\text{-step explicit linear multistep formula of order } p\ge 2 
\text{ with } \sspcoeff_n(\omega,\delta,\beta)>0. 
\end{equation}
We prove that there is a $k$-step formula of order $p$ whose SSP coefficient is equal to the optimal 
one.

Let us set $\Hh:=\{r> 0 : \text{the statement } (ii) \text { of Lemma \ref{lem:farkas}\, holds}\}$.
Step 1 implies that the largest possible 
$\sspcoeff_n(\omega,\delta,\beta)$ for all $k$-step formulae of order $p$
and having the given time-step ratios $\omega_j$ is finite. So
$\Hh\ne\varnothing$ by Lemma~\ref{lem:farkas}. We also see that
$\Hh$ is an infinite interval, because if a polynomial $q_n$ satisfies
 the conditions \eqref{dualconditions:a}, \eqref{dualconditions:b} with some $r := \rho> 0$, and
 \eqref{dualconditions:c:farkas}, then it satisfies the same
 conditions with any $r \in (\rho, +\infty)$. Thus, with a suitable $r^*\ge 0$, we have
$(r^*,+\infty)\subseteq \Hh\subseteq [r^*,+\infty)$. 
Clearly, $r^*>0$ due to assumption \eqref{Thm2Step2assumption} and Lemma~\ref{lem:farkas}.
We claim that $r^*\notin \Hh$.

Suppose to the contrary that $r^*\in \Hh$. Then there is a polynomial $q_n$
 satisfying the conditions \eqref{dualconditions:a}, \eqref{dualconditions:b} with $r := r^*$, and
 \eqref{dualconditions:c:farkas}. Now we define 
 $\widetilde{q}_n := q_n+\left| q_n(\Omega_{k,n})\right|/2$. One easily checks that this 
$\widetilde{q}_n$  polynomial satisfies the same set of conditions with 
\[
r:=\left(\frac{1+\max\limits_{0\le j\le k-1} q_n(\Omega_{j,n})}
{1+{\left|q_n(\Omega_{k,n})\right|}/{2}+\max\limits_{0\le j\le k-1} q_n(\Omega_{j,n})} \right)
\cdot r^* \in (0, r^*),
\]
so we would get $(0, r^*)\cap \Hh \ne \varnothing$, a contradiction.

Hence $\Hh= (r^*,+\infty)$. Therefore, in view of Lemma~\ref{lem:farkas}, $r^*$ is the optimal SSP coefficient and there exists an optimal method with $\sspcoeff_n(\omega, \delta, \beta) = r^*$.
\end{proof}

\subsection{The proof of Theorem \ref{thm:optimal2ndorder}}\label{sec:proof0Thm3proof}
In the following we apply the usual terminology and say that an inequality constraint 
is \textit{binding} if the inequality holds with equality. The lemma below
is used in the proofs of Theorems \ref{thm:optimal2ndorder} and \ref{thm:optimal3rdorder}.
\begin{lem}
  \label{lem:dualoptimal}
  Let $n\ge k$ be arbitrary positive integers and $p\ge 2$.
  Suppose that the time-step ratios $\omega_{j,n}>0$ are given, and that there
  exists an explicit linear multistep formula with $k$ steps, order of accuracy $p$, and
  $\sspcoeff_n>0$.
Let $\delta_{j,n}$ and $\beta_{j,n}$ denote the coefficients of a formula
with the largest SSP coefficient $\sspcoeff_n(\omega_n,\delta_n,\beta_n)\in  (0,+\infty)$. 
  Let $q_n$ be a polynomial that satisfies all the conditions of Lemma~\ref{lem:dual}
  with $r:=\sspcoeff_n(\omega_n,\delta_n,\beta_n)$.
  Then the following statements hold.
  \begin{enumerate}
    \item[(i)] If $\delta_{j,n} \ne 0$ for some $0\le j\le k-1$, then
        $q_n(\Omega_{j,n}) = 0$.
    \item[(ii)] If $\beta_{j,n} \ne 0$ for some $0\le j\le k-1$, then
        $q_n'(\Omega_{j,n})+rq_n(\Omega_{j,n}) = 0$.
    \item[(iii)]
      This $q_n$ can be chosen so that the total number of binding inequalities
      in \eqref{dualconditions:a}-\eqref{dualconditions:b} is at least $p$. 
  \end{enumerate}
\end{lem}

\begin{proof}[The proof of Theorem \ref{thm:optimal2ndorder}]
  The necessity of the condition $\Omega_{k,n} > 2$ for the existence of a second-order formula
  with positive SSP coefficient is an immediate consequence
  of Theorem~\ref{thm:upperbound}.
  On the other hand, we easily see that, up to a positive multiplicative constant, $q_n(x) = x(\Omega_{k,n}-x)$
is the unique 
  polynomial satisfying \eqref{dualconditions:a}, \eqref{dualconditions:c}, 
  and Properties 1 and 4 of Lemma~\ref{lem:dual}. If $\Omega_{k,n} > 2$, then $q_n$ does not
  satisfy \eqref{dualconditions:b} with $r:=0$ for $j=k-1$, therefore there exists a formula with $\sspcoeff_n > 0$ due to Lemma~\ref{lem:dual}.
  Thus the condition $\Omega_{k,n} > 2$ is also sufficient.

From now on we assume $\Omega_{k,n} > 2$, and that a formula with the optimal SSP
coefficient $\sspcoeff_n(\omega_n, \delta_n, \beta_n)$ is considered.
 From the above form of $q_n$ we see
that $q_n(\Omega_{j,n}) = 0$ ($0\le j\le k-1$) holds only for $j=0$, so the statement $(iii)$ of Lemma~\ref{lem:dualoptimal}
  guarantees that at least one of the inequalities in \eqref{dualconditions:b} is binding with $r:=\sspcoeff_n(\omega_n, \delta_n, \beta_n)$.
  Since $q_n'(\Omega_{0,n})+r q_n(\Omega_{0,n})=\Omega_{k,n}>0$
and for all $1 \le j \le k-2$ we have
\[
    \frac{q_n'(\Omega_{j,n})}{q_n(\Omega_{j,n})} = \frac{1}{\Omega_{j,n}}-\frac{1}{\Omega_{k,n}-\Omega_{j,n}}
    > \frac{1}{\Omega_{k-1,n}}-\frac{1}{\Omega_{k,n}-\Omega_{k-1,n}} = -\frac{\Omega_{k-1,n}-1}{\Omega_{k-1,n}},
\]
the unique binding inequality in \eqref{dualconditions:b} must be the one corresponding to the index $j=k-1$. But $q_n'(\Omega_{k-1,n})+r q_n(\Omega_{k-1,n}) =0$ implies 
$r=\left({\Omega_{k,n}-2}\right)/\left({\Omega_{k,n}-1}\right)$.

  Finally, the statements $(i)$ and $(ii)$ of Lemma \ref{lem:dualoptimal} guarantee that
  only $\delta_{0,n}$ and $\beta_{k-1,n}$ can differ from zero.  Solving the conditions for
  order two (see \eqref{oc-vss-delta}) then yields the unique formula stated in 
\eqref{VSSSSPLMMk2}.
\end{proof}

\subsection{The proof of Theorem \ref{thm:optimal3rdorder}}\label{sec:proof0Thm4proof}

First we present a lemma that 
%supersedes Lemma \ref{rkjdeflemma}, and 
will be used in 
the proof of Theorem \ref{thm:optimal3rdorder}. Moreover, Statement 2 of the lemma proves our
claim about the unique real root of the polynomial $P_{k+j}$ in \eqref{Pkjdef} under assumption 
\eqref{22cond}, hence guaranteeing that the quantity $r_{k+j}$ in \eqref{rhokjdef} has a 
proper definition. 
The $\Delta_{j}$ quantities defined in
\eqref{Deltadef} clearly satisfy the condition $1<\Delta_{j+1}<\Delta_{j}$ below.
\begin{lem}\label{lemma5} Let some $1<\Delta_{j+1}<\Delta_{j}$ numbers be given, and for any 
$\rho\in\mathbb{R}$ let
\begin{equation}\label{Ajmatrixdef}
A_{j}(\rho):=
 \begin{pmatrix}
                    \rho-2&\rho-1&\rho-3\\
                    \rho \Delta_{j}^2-2 \Delta_{j} &\rho \Delta_{j}-1 &\rho \Delta_{j}^3-3 \Delta_{j}^2\\
                    \rho \Delta_{j+1}^2-2 \Delta_{j+1}  & \rho \Delta_{j+1}-1 & \rho \Delta_{j+1}^3-3 \Delta_{j+1}^2
\end{pmatrix}
\begin{matrix}\vphantom{}\\ \vphantom{}\\\vphantom{}.\end{matrix}
\end{equation}
Then the following statements hold.
\begin{enumerate}
\item Every real root of the polynomial $P_{k+j}$ (appearing in \eqref{Pkjdef}) is positive.
\item Under condition \eqref{22cond}, $P_{k+j}$ has a unique real root.
\item If 
\begin{equation}\label{lemma5up}
\exists (r, a, b)\in\mathbb{R}^3 \text{ satisfying }
A_{j}(r)\cdot(a,b,1)^\top=0 \text{ and } a^2-4b<0,
\end{equation}
then
\begin{equation}\label{lemma5down}
\eqref{22cond}\quad \text{and}\quad P_{k+j}(r)=0
\end{equation}
hold.
\item If \eqref{lemma5down} holds with some $r\in\mathbb{R}$, then \eqref{lemma5up} is true, 
and the triplet $(r, a, b)\in\mathbb{R}^3$ is unique.
\end{enumerate}
\end{lem}
\begin{figure}
\begin{center}
\includegraphics[width=0.5\textwidth]{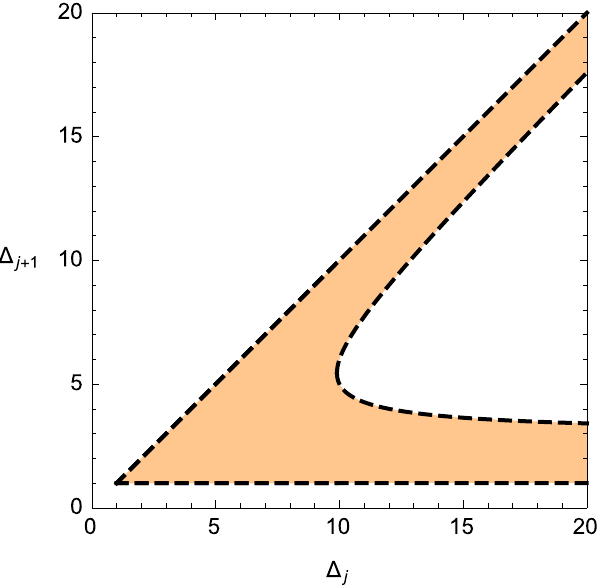}
\caption{The set defined by \eqref{22cond}.
\label{deltajfigure}}
\end{center}
\end{figure}
\begin{proof} Throughout the proof we always assume $1<\Delta_{j+1}<\Delta_{j}$ and $r\in\mathbb{R}$. First we define some auxiliary polynomials---their dependence on 
$\Delta _{j+1}$ is suppressed for brevity.
\begin{description}[labelindent=-4pt]
\item $\widetilde{P}_1(r):=(r^2-r)\Delta_{j+1}-r+2$,
\item $\widetilde{P}_2(r):=(r^2-2r)\Delta_{j+1}-2 r+6$,
\item $\widetilde{P}_3(r):=(r^2-r^3) \Delta _{j+1}^2+2 (r^2-2r) \Delta_{j+1}-2 r+6$,
\item $\widetilde{P}_4(r):=r^2 \left(r^2-4 r+2\right) \Delta _{j+1}^2-4 r \left(r^2-5 r+3\right) \Delta _{j+1}+2\left(r^2-6 r+6\right)$, 
\item $\widetilde{P}_5(r):=\left(-\left(5+2 \sqrt{6}\right) r^3+2 \left(3+\sqrt{6}\right) r^2-2r\right) \Delta_{j+1}+\left(5+2 \sqrt{6}\right) r^2-4 \left(3+\sqrt{6}\right) r+6$,
\item $\widetilde{P}_6(r):=(r-1) \Delta_{j+1}-r+3$,
\item $\widetilde{P}_7(r):=r^2\Delta_{j+1}^2-4 r \Delta_{j+1}+6$,
\item $\widetilde{P}_8(x, a, r):=(r^2-r) x^2+(r-1) (a r+r-3)x+a (2-r)-r+3$,
\item $\widetilde{P}_9(r):=(r-r^2) \Delta_{j+1}^2-\left(r^2-4 r+3\right) \Delta_{j+1}+r-3$,
\item $\widetilde{P}_{10}(r):=(r^2-r) \Delta_{j+1}-r+3$,
\item $\widetilde{P}_{11}(r):=r \Delta_{j+1}^2-(r+3) \Delta_{j+1}+1$.
\end{description}

\indent \textbf{Step 1.} Statement 1 follows from the fact that the coefficient of $x^m$ 
in $P_{k+j}(x)$ is positive for odd $m$ and negative for even $m$, so $P_{k+j}$ cannot have a non-positive root.

 \indent \textbf{Step 2.} We now prove Statement 2. Since $P_{k+j}$ is cubic, it has a real root. By taking into account Statement 1, we will show that
\begin{equation}\label{lem51}
\exists r>0\ : \  \eqref{22cond},\   P_{k+j}(r)=0\  \text{ and } \ P_{k+j}'(r)\le 0 
\end{equation}
is \textit{false}, implying that $P_{k+j}'>0$ at every real root of $P_{k+j}$, 
which is possible only if $P_{k+j}$ has a unique real root.

If $\widetilde{P}_1(r)=0$,  we obtain 
\begin{equation}\label{rdeltaj147}
(r, \Delta_{j+1})=(2-\sqrt{2}, 3+2 \sqrt{2})
\end{equation}
by using $P_{k+j}(r)=0$ also. But then one easily shows that \eqref{lem51} is impossible.

So we can suppose in the rest of Step 2 that $\widetilde{P}_1(r)\ne 0$. Then from $P_{k+j}(r)=0$ we get 
\begin{equation}\label{Deltajexpression47}
\Delta_j=\widetilde{P}_2(r)/({r\widetilde{P}_1(r)}).
\end{equation}
This form of $\Delta_j$ is substituted into \eqref{lem51} and we obtain
\begin{equation}\label{system47}
 r>0\quad\text {and}\quad 
{\widetilde{P}_3(r)}/{\widetilde{P}_1(r)}>0
\quad\text {and}\quad
\widetilde{P}_4(r)/\widetilde{P}_1(r)\le 0 \quad\text {and}
\end{equation}
\begin{equation}\label{system48}
\left( {\widetilde{P}_5(r)}/{{\widetilde{P}_1(r)}} <0 \quad\text {or}\quad  
{\widetilde{P}_6(r)\,\widetilde{P}_7(r)}/{\widetilde{P}_1(r)}>0 \right).
\end{equation}
Because all $\widetilde{P}_j$ polynomials ($1\le j\le 7$) are at most quadratic in $\Delta_{j+1}$, 
one can systematically reduce \eqref{system47}-\eqref{system48} to a system of univariate
polynomial inequalities in $r$, and verify that  \eqref{system47}-\eqref{system48} has no solution. This finishes the proof of Step 2.

\indent \textbf{Step 3.} To prove Statement 3, notice first that 
\begin{equation}\label{vector49}
A_{j}(r) \cdot (a,b,1)^\top = 0
\end{equation}
 implies 
\begin{equation}\label{Pkjzero}
0=\det A_{j}(r)\equiv \left(\Delta_j-1\right) \left(\Delta_{j+1}-1\right) \left(\Delta_j-\Delta_{j+1}\right) P_{k+j}(r),
\end{equation}
so we have $P_{k+j}(r)=0$. To show \eqref{22cond}, 
depicted in Figure \ref{deltajfigure}, we separate two cases. 

If $r=1$, then we obtain $a=-2$ from the first component of  \eqref{vector49}, and
$b=4 \Delta_j-\Delta_j^2$ with 
\begin{equation}\label{delta50}
\Delta_{j+1}=4-\Delta_j\quad\text{ and }\quad 2<\Delta _j<3
\end{equation}
from the other two components and from $1<\Delta_{j+1}<\Delta_{j}$. So 
 \eqref{22cond} holds.

We can thus suppose in the rest of Step 3 that $r\ne 1$. Then from \eqref{vector49} we get
\begin{equation}\label{b52eq}
b=\frac{(2-r)a-r+3}{r-1}
\end{equation}
and
\begin{equation}\label{P851}
\widetilde{P}_8(\Delta_{j}, a, r)=0=\widetilde{P}_8(\Delta_{j+1}, a, r).
\end{equation}
We again separate two cases.

If $\widetilde{P}_1(r)=0$, then by using \eqref{P851} and $1<\Delta_{j+1}<\Delta_{j}$ we obtain
\begin{equation}\label{ineq52}
r=2-\sqrt{2}, \  \Delta _j>\Delta _{j+1}=3+2 \sqrt{2} \quad\text{and}\quad 
   a=-\left(\Delta_j+1+{1}/{\sqrt{2}}\right).
\end{equation}
But it is easy to check that \eqref{ineq52}, the inequality $a^2-4b<0$ from \eqref{lemma5up}, and the \textit{negation} of \eqref{22cond} cannot hold simultaneously.
So \eqref{22cond} is proved in this case.

Therefore, in the rest of Step 3, we can suppose that $\widetilde{P}_1(r)\ne 0$. Then $a$ is expressed
from $\widetilde{P}_8(\Delta_{j+1}, a, r)=0$ and we get
\begin{equation}\label{step3aform}
a={\widetilde{P}_9(r)}/{\widetilde{P}_1(r)}.
\end{equation}
Moreover, we express $\Delta_j$ from $P_{k+j}(r)=0$ (shown to hold in \eqref{Pkjzero}) as
in \eqref{Deltajexpression47}. To finish Step 3, we will verify that 
\begin{equation}\label{finalstep3}
a^2-4b<0 \text{ from } \eqref{lemma5up}, \text{ and the negation of } \eqref{22cond}
\end{equation}
cannot hold. The substitution of \eqref{step3aform} and \eqref{Deltajexpression47} into 
\eqref{finalstep3} yields
\begin{equation}\label{final2step3}
r>0, \quad \widetilde{P}_{6}(r) \widetilde{P}_{10}(r) \widetilde{P}_{11}(r)<0,
\quad \widetilde{P}_{5}(r)/\widetilde{P}_{1}(r)\ge 0,
\quad\text{and}\quad \widetilde{P}_{6}(r)\widetilde{P}_{7}(r)/\widetilde{P}_{1}(r)\le 0.
\end{equation}
Again, these $\widetilde{P}_{j}$ polynomials are at most quadratic in $\Delta_{j+1}$, 
so one can check that \eqref{final2step3} has no solution.

\indent \textbf{Step 4.} To prove Statement 4, we first suppose that 
\begin{equation}\label{58cond}
r\ne 1\quad\text{and}\quad \widetilde{P}_{1}(r)\ne 0.
\end{equation} 
Then we have seen in Step 3 that
if $(a,b)\in\mathbb{R}^2$ solves \eqref{vector49}, then $a$ and $b$ necessarily have the 
form \eqref{step3aform} and \eqref{b52eq}, respectively; and the uniqueness of $r$
is guaranteed by Step 2. So, 
under condition \eqref{58cond}, any triplet solving \eqref{vector49} is unique. On the other hand,
 direct
substitution into \eqref{vector49} shows that the above $(r, a, b)$ is indeed a 
solution---the non-trivial component to check in \eqref{vector49} is the second one,
which is just   
$\left(\Delta_j-1\right) \left(\Delta_j-\Delta_{j+1}\right) P_{k+j}(r)/
\widetilde{P}_{1}(r)=0$ due to $ P_{k+j}(r)=0$ from assumption \eqref{lemma5down}.
Finally, to finish Step 4 in the case when \eqref{58cond} holds, we notice that for this unique triplet
$(r, a, b)$ we have
$a^2-4b=\widetilde{P}_{6}(r) \widetilde{P}_{10}(r) \widetilde{P}_{11}(r)/(\widetilde{P}_{1}(r))^2$. 
But, with $\Delta_j$ expressed as in \eqref{Deltajexpression47}, one can again show that
\[
\widetilde{P}_{6}(r) \widetilde{P}_{10}(r) \widetilde{P}_{11}(r)\ge 0, \quad 
1<\Delta_{j+1}<\Delta_{j}, \quad\text{and}\quad
\eqref{22cond}\text{ from assumption }\eqref{lemma5down}
\]
cannot hold. This means that $a^2-4b<0$ from \eqref{lemma5up} also holds.

To prove Statement 4 when $r=1$, we recall from Step 3 that in this case the only possible
triplet is $(r, a, b)= (1, -2, 4 \Delta_j-\Delta_j^2)$ with \eqref{delta50}, which indeed satisfy
\eqref{lemma5up}.

Lastly, we consider the case when $\widetilde{P}_{1}(r)=0$. We have seen in Step 2 that this time 
\eqref{rdeltaj147} holds. Since $\Delta_{j+1}<\Delta_{j}$, we also have 
$3+2 \sqrt{2}<\Delta_{j}$. These yield that the unique solution to \eqref{vector49} is 
\[
(r, a, b) = (2-\sqrt{2},\  -(\Delta_j+1+{1}/{\sqrt{2}}), \ (2+\sqrt{2}) \Delta_j).
\]
Now \eqref{22cond} implies $\Delta_j<5+{7}/{\sqrt{2}}$, so 
$a^2-4b=\Delta_j^2-3(2+\sqrt{2}) \Delta_j+\sqrt{2}+{3}/{2}<0$, and the proof is complete.
\end{proof}

\begin{proof}[The proof of Theorem \ref{thm:optimal3rdorder}] 

\textbf{Step 1.}  The necessity of the condition $\Omega_{k,n} > 3$ for the existence of a third-order formula with positive SSP coefficient
  is again an immediate consequence of Theorem~\ref{thm:upperbound}. 
  To see that this condition is also sufficient, suppose to the contrary that there is no third-order
  formula with $k$ steps, positive SSP coefficient and $\Omega_{k,n} > 3$, and let us apply Lemma~\ref{lem:dual} with $r:=0$.
  First we show that, up to a positive multiplicative constant, $q_n(x) = x^2(\Omega_{k,n}-x)$ is 
the unique 
   polynomial satisfying \eqref{dualconditions}, and
  Properties 1--2 and 4--5 of Lemma~\ref{lem:dual}.
  Indeed, \eqref{dualconditions:a}--\eqref{dualconditions:b} with $j=0$ and $r=0$,
  together with Properties 2 and 5 imply that $\Omega_{0,n}=0$ is a root of $q_n$
  (otherwise the logarithmic derivative of $q_n$ at $0$ would be negative), and
  then the multiplicity of $0$ is precisely two due to Properties 1 and 4, so the desired form
of $q_n$ follows. 
  Now, applying \eqref{dualconditions:b} with $j=k-1$ and $r=0$ to this $q_n$
  gives 
$\left(3-\Omega_{k,n}\right) \left(\Omega_{k,n}-1\right)\ge 0$, contradicting to 
$\Omega_{k,n} > 3$. This contradiction proves the sufficiency of the condition $\Omega_{k,n} > 3$.

 \indent \textbf{Step 2.}  From now on we assume $\Omega_{k,n} > 3$, and that a formula with the optimal
  SSP coefficient $r:=\sspcoeff_n(\omega_n,\delta_n,\beta_n)\in  (0,+\infty)$ is considered. 
  Let $q_n$ be a polynomial satisfying conditions \eqref{dualconditions},
  Properties 1--6 of Lemma \ref{lem:dual}, and the condition $(iii)$
  of Lemma~\ref{lem:dualoptimal}.

 \indent \textbf{Case I.} In the case when $\Omega_{k,n}$ is not the only real root of $q_n$,  
then, due to Properties 1, 3 and 5 of Lemma~\ref{lem:dual}, $q_n$ is unique (up to a positive multiplicative constant)  and 
  takes the form 
  $q_n(x)=(x-\Omega_{j_0,n})^2(\Omega_{k,n}-x)$ with an appropriate
  index $0\leq j_0\leq k-1$. We notice that 
the only binding inequality in \eqref{dualconditions:a} is the one
corresponding to $j=j_0$, and the $j=j_0$ inequality in \eqref{dualconditions:b} is also binding 
(independently of the value of $r$). 
On the other hand, now all roots of $q_n$ are real, so the logarithmic derivative of $q_n$ is strictly 
  decreasing on the intervals $[\Omega_{0,n},\Omega_{j_0,n})$ and $(\Omega_{j_0,n},\Omega_{k-1,n}]$.
Thus, by $(iii)$ of Lemma~\ref{lem:dualoptimal}, 
\eqref{dualconditions:b} contains either two or three binding inequalities,
and they correspond to
\begin{itemize}
  \item[\textbf{A.}] $j\in\{0, k-1\}$ for $j_0=0$,
  \item[\textbf{B.}] $j\in\{j_0-1, j_0\}$ or $j\in \{j_0, k-1\}$ for $1\le j_0\le k-2$,
  \item[\textbf{C.}] $j\in \{k-2, k-1\}$ for $j_0=k-1$.
\end{itemize}
In Case A, \eqref{dualconditions:b} is solved with $j=k-1$,
and we obtain $r={(\Omega_{k,n}-3)}/{(\Omega_{k,n}-1)}$.  Due to $(i)$ and $(ii)$ of Lemma~\ref{lem:dualoptimal}, all formula coefficients except for possibly
$\{\delta_{0,n}, \beta_{0,n}, \allowbreak \beta_{k-1,n}\}$ vanish.\\
In Case B, \eqref{dualconditions:b} is solved with $j=j_0-1$ and $j=k-1$ to get
$r={2}/{(\Omega_{j_0,n}-\Omega_{j_0-1,n})}+{1}/({\Omega_{k,n}-\Omega_{j_0-1,n}})$ and 
$r={(\Omega_{k,n}-\Omega_{j_0,n}-3)}/{(\Omega_{k,n}-\Omega_{j_0,n}-1)}$, respectively, then the maximum is chosen. If these two expressions for $r$ are not equal, we get from 
$(i)$ and $(ii)$ of Lemma~\ref{lem:dualoptimal} that the number of 
non-zero formula coefficients is at most three; whereas if the two expressions for $r$ are equal,
the number of non-zero formula coefficients is at most four.\\
In Case C, \eqref{dualconditions:b} is solved with $j=k-2$
to yield $r={2}/{(\Omega_{k-1,n}-\Omega_{k-2,n})}+{1}/{(\Omega_{k,n}-\Omega_{k-2,n})}$.
All formula coefficients except for possibly 
$\{\delta_{k-1,n}, \beta_{k-2,n}, \beta_{k-1,n}\}$ vanish.

Therefore, the optimal SSP coefficient is equal to one of the $r_j$ quantities ($0\le j\le k-1$)
defined in Section \ref{sec:3rdordermethods}, and the 
non-zero formula coefficients also have the form stated there. 

\indent \textbf{Case II.}  In the case when $\Omega_{k,n}$ is the only real root of $q_n$, 
then, by taking into account \eqref{dualconditions:a}, and Properties 1 and 5 of Lemma~\ref{lem:dual}, $q_n$ (up to a positive multiplicative constant) has the form
\[q_n(x)=(\Omega_{k,n}-x)\left((\Omega_{k,n}-x)^2+a(\Omega_{k,n}-x)+b\right)\]
with some coefficients $a, b\in \mathbb{R}$ satisfying the discriminant condition
$a^2-4b<0$. Let us introduce the abbreviation $Q_3:=q_n' + rq_n$.
 
This time there are no binding inequalities in \eqref{dualconditions:a}, 
so, due to $(iii)$ of Lemma~\ref{lem:dualoptimal}, 
 precisely three inequalities
   are binding  in \eqref{dualconditions:b}---there can be no more, since
the polynomial $Q_3$ 
is cubic. Let $0\le j_0< j_1< j_2\le k-1$ denote the three indices corresponding to these binding inequalities, then
$Q_3(\Omega_{j_0,n})=Q_3(\Omega_{j_1,n})=Q_3(\Omega_{j_2,n})=0$.
 We observe that 
the leading coefficient of $Q_3$ is $-r<0$, and 
$Q_3(\Omega_{k,n}) = -b <0$ because of $a^2-4b<0$. 
As $\Omega_{j,n}$ is strictly increasing in $j$, we get that
 $Q_3$ is positive on 
$(-\infty, \Omega_{j_0,n})\cup (\Omega_{j_1,n}, \Omega_{j_2,n})$, and negative on 
$(\Omega_{j_0,n}, \Omega_{j_1,n})\cup (\Omega_{j_2,n}, +\infty)$. Hence neither
$j_1\ge j_0+2$ nor $j_2\le k-2$ can occur (otherwise $Q_3(\Omega_{j_0+1,n})<0$ 
or $Q_3(\Omega_{k-1,n})<0$ would contradict to \eqref{dualconditions:b}). 
This shows that the binding inequalities in \eqref{dualconditions:b} are the ones corresponding to the index set $j\in\{j_0, j_0+1, k-1\}$ with an appropriate index $0\leq j_0\leq k-3$.
The previous sentence can be rewritten as 
$A_{j_0}(r)\cdot(a,b,1)^\top=0$, with the matrix $A_{j_0}$ defined in \eqref{Ajmatrixdef}. Since 
now $a^2-4b<0$ as well, so \eqref{lemma5up} is also satisfied. Hence \eqref{22cond} and $P_{k+j_0}(r)=0$ hold by
Lemma \ref{lemma5}. Due to  
uniqueness we see that $r=r_{k+j_0}$, with $r_{k+j_0}$ defined in \eqref{rhokjdef}.

Thus, the optimal SSP coefficient is equal to one of the $r_{k+j_0}$ quantities ($0\le j_0\le k-3$),
and we get from $(ii)$ of Lemma~\ref{lem:dualoptimal}  that  
all formula coefficients except for possibly $\{\beta_{j_0,n}, \beta_{j_0+1,n}, \beta_{k-1,n}\}$
  vanish.
\begin{sloppypar}
\textbf{Step 3.} Finally, in order to conclude that the optimal SSP coefficient
$\sspcoeff_n(\omega_n,\delta_n,\beta_n)$
is in fact the minimum of the $r_j$ expressions defined in Section \ref{sec:3rdordermethods},
we show that $\sspcoeff_n(\omega_n,\delta_n,\beta_n)\le r_j$ for any $0\le j\le 2k-3$.
\end{sloppypar}

Indeed, fix any $0\leq j_0 \leq k-1$ and define $q_n(x):=(x-\Omega_{j_0,n})^2(\Omega_{k,n}-x)$. 
The inequalities \eqref{dualconditions:a} and \eqref{dualconditions:c} are trivially
satisfied. We verify \eqref{dualconditions:b} with $r := r_{j_0}$ by 
setting $\widetilde{Q}_3:=q_n' + rq_n$ and distinguishing three cases. 
\begin{itemize}[leftmargin=33pt]
  \item[\textbf{A.}] If $j_0=0$, then one checks that the three roots of $\widetilde{Q}_3$ are
found at $\{x_1, 0, \Omega_{k-1,n}\}$ with $x_1=-{2 \Omega_{k,n}}/\left({\Omega_{k,n}-3}\right)<0$.
  \item[\textbf{B.}] If $1\le j_0\le k-2$ and 
\begin{equation}\label{lhsrhs}
\frac{\Omega_{k,n}-\Omega_{j_0,n}-3}{\Omega_{k,n}-\Omega_{j_0,n}-1}<
\frac{2}{\omega_{j_0,n}}+\frac{1}{\Omega_{k,n}-\Omega_{j_0-1,n}},
\end{equation}
then the three roots of $\widetilde{Q}_3$ are
found at $\{\Omega_{j_0-1,n}, \Omega_{j_0,n}, x_3\}$
with $\Omega_{k-1,n}<x_3<\Omega_{k,n}$. If we have ``$>$'' in \eqref{lhsrhs}, then
the three roots of $\widetilde{Q}_3$ are
found at $\{x_1, \Omega_{j_0,n}, \Omega_{k-1,n}\}$
with $\Omega_{j_0-1,n}<x_1<\Omega_{j_0,n}$. Finally, if
there is ``$=$'' in \eqref{lhsrhs}, then the roots of $\widetilde{Q}_3$ are
found at $\{\Omega_{j_0-1,n}, \Omega_{j_0,n}, \Omega_{k-1,n}\}$.
  \item[\textbf{C.}] If $j_0=k-1$, then the three roots of $\widetilde{Q}_3$ are
found at $\{\Omega_{k-2,n}, \Omega_{k-1,n}, x_3\}$ with $\Omega_{k-1,n}<x_3<\Omega_{k,n}$.
\end{itemize}
These, combined with the fact that the leading coefficient of $\widetilde{Q}_3$ is negative, mean that  all inequalities in \eqref{dualconditions} are satisfied.
Thus, in view of Lemma \ref{lem:dual}, $\sspcoeff_n\le r_{j_0}$.

Now fix an arbitrary index $0\le j_0\le k-3$. The inequality $\sspcoeff_n\le r_{k+j_0}$
is trivial if $r_{k+j_0}=+\infty$. Otherwise, if $r_{k+j_0}<+\infty$ in \eqref{rhokjdef}, then, 
by definition, \eqref{22cond} and $P_{k+j_0}(r_{k+j_0})=0$ hold. So due to Lemma \ref{lemma5},
\eqref{lemma5up} holds with $r:=r_{k+j_0}$ and $j:=j_0$. By defining 
 \[q_n(x):=(\Omega_{k,n}-x)\left((\Omega_{k,n}-x)^2+a(\Omega_{k,n}-x)+b\right)\]
and with $a$ and $b$ given by \eqref{lemma5up}, we see that \eqref{dualconditions:a} and \eqref{dualconditions:c} are true because of $a^2-4b<0$.  On the other hand, the 
$A_{j_0}(r)\cdot(a,b,1)^\top=0$ relation in \eqref{lemma5up} expresses the fact that
the inequalities corresponding to 
$j\in\{j_0, j_0+1, k-1\}$ are binding in \eqref{dualconditions:b}. So, just as in Case II in Step 2,
we see that $q_n' + r q_n$ is positive on 
$(-\infty, \Omega_{j_0,n})\cup (\Omega_{j_0+1,n}, \Omega_{k-1,n})$, and negative on 
$(\Omega_{j_0,n}, \Omega_{j_0+1,n})\cup (\Omega_{k-1,n}, +\infty)$. Thus all inequalities
in \eqref{dualconditions:b} hold.
This means that $q_n$ satisfies conditions \eqref{dualconditions} with this $r$ value, 
so, in view of Lemma \ref{lem:dual}, $\sspcoeff_n\le r_{k+j_0}$.
 The proof of the theorem is complete.
\end{proof}

\subsection{The proof of Theorem \ref{sec3Thm5}}\label{sec:proof0Thm5proof}

\begin{proof}
We follow the ideas of the proof given in \cite{LMMpaper} for the fixed-step-size case.

Suppose to the contrary that $p$,  $k$ and $\omega_{j,n}$ ($1\le j \le k$)
satisfy the conditions of the theorem for some $n \ge k$ and $K_1, K_2 \ge 1$, but for all formulae
with $k$ steps and order of accuracy $p$, we have $\sspcoeff_n(\omega, \delta, \beta) = 0$.
Then by Lemma~\ref{lem:dual}, there exists a non-zero real polynomial $q_n$  that satisfies the conditions \eqref{dualconditions:a}--\eqref{dualconditions:c}
with $r=0$, moreover $q_n\ge 0$ on $[0,\Omega_{k,n}]$ and $\deg q_n=p$
(Properties 1 and 6 of Lemma \ref{lem:dual}). 

First we define $A := \max_{x \in [0,\Omega_{k,n}]} q_n(x)$
and $b := \max_{x \in [0,\Omega_{k,n}]} |q_n'(x)|$ and we introduce
the polynomial $P(x) := q_n\left(\frac{x+1}{2}\cdot\Omega_{k,n}\right)-\frac{A}{2}.$
Then the Markov brothers' inequality (see, e.g., \cite{LMMpaper}) for the first derivative implies that
\begin{equation*}
  \max_{x\in [-1,1]}|P'(x)| \le p^2 \cdot \max_{x\in [-1,1]}|P(x)|,
\end{equation*}
that is, $b\frac{\Omega_{k, n}}{2} \le p^2\frac{A}{2}$.
On the other hand, summing the lower estimates in \eqref{K1K2omegajineq} we get 
$\frac{k}{K_1} \le \Omega_{k, n}$, implying 
$\frac{b}{2}\cdot \frac{k}{K_1} \le \frac{b}{2}{\Omega_{k, n}}$. Thus 
\begin{equation}\label{bkK1ineq}
{b k}/{K_1}\le p^2 A.
\end{equation}

Now, because of $q_n(\Omega_{k,n}) = 0$, the Newton--Leibniz formula and elementary estimates yield that
\begin{equation}\label{maxqprime}
A \le \int_0^{\Omega_{k,n}} \max(0,-q_n'(t)) dt.
\end{equation}
Here we notice that the polynomial $q_n'$ is of degree at most $p-1$, and $q_n'(\Omega_{j,n}) \ge 0$
for all $0 \le j\le k-1$, so the set \[
\{j\in \mathbb{Z}\cap [0,k-1] : \exists x \in [\Omega_{j,n}, \Omega_{j+1, n}] 
\text{ with } q_n'(x) < 0 \}
\]
has at most ${p}/{2}$ elements. Therefore---by decomposing the interval $[0,{\Omega_{k,n}}]$ as
the union of the appropriate subintervals of length $\omega_{j,n}$, and applying the upper estimate
in \eqref{K1K2omegajineq} and the estimate
$-q_n'\le b$ at most $p/2$ times---we get that 
$\int_0^{\Omega_{k,n}} \max(0,-q_n'(t)) dt\le \frac{p}{2}K_2 b$, and hence
\begin{equation}\label{ApK2bineq}
A\le {b p K_2}/{2}.
\end{equation}

Inequalities \eqref{bkK1ineq} and \eqref{ApK2bineq} imply that $k \le p^3 K_1 K_2/2$,
which contradicts the assumption of Theorem \ref{sec3Thm5}. Hence there is a formula with $k$ steps,
order of accuracy $p$, and $\sspcoeff_n(\omega,\delta,\beta) > 0$.
\end{proof}

\section{The proofs of the theorems in Section \ref{sec:stepsize}\label{sec:proofs}}

In Section \ref{globattrsction} we first prove a theorem about the 
convergence of some rational recursions.  This Theorem 
\ref{attractivitytheorem} will then be used in Sections \ref{proofofTh2section} and
\ref{proofofTh3section} to prove Theorems \ref{hnconvergencetheorem} 
and \ref{hnconvergencetheorem3rdorder}, respectively.

\subsection{Global attractivity in a class of higher-order rational recursions}\label{globattrsction}

\begin{thm}\label{attractivitytheorem}
Let us fix an integer $k\ge 3$ and a real number $A>0$. Suppose that for $1\le j \le k-1$
the initial values $\tau_j\ge 0$ are given such that $\sum_{j=1}^{k-1}\tau_j>0$.
For any $n\ge k$ we define
\begin{equation}\label{taudef}
\tau_n:=\frac{\sum_{j=1}^{k-1}\tau_{n-j}}{A+\sum_{j=1}^{k-1}\tau_{n-j}}.
\end{equation}
Then 
\[
\lim_{n\to +\infty} \tau_n=\left\{
\begin{aligned}
    &  0, & \text{ if }\quad  & k-1\le A, \\ 
   & \frac{k-1-A}{k-1},  & \text{ if }\quad  & 0<A<k-1.
\end{aligned}
\right.
\]
\end{thm}
To prove Theorem 
\ref{attractivitytheorem}, we will apply the following lemma. 
\begin{lem}[Theorem A.0.1 in \cite{KulenovicLadas}]\label{ThA01Lemma}
Suppose that $a\le b$ are given real numbers, $k\ge 3$ is a fixed integer,
and the numbers $x_j$ are chosen such that $x_j\in [a,b]$ for $1\le j\le k-1$. Assume further
 that 
\begin{enumerate}
\item $f\colon [a,b]^{k-1}\to [a,b]$ is continuous,
\item $f$ is non-decreasing in each of its arguments,
\item there is a unique $\overline{x}\in [a,b]$ such that $f(\overline{x},\overline{x},\ldots, \overline{x})=\overline{x}$,
\item and the sequence $x_n$ is defined for $n\ge k$ as 
\[
x_n:=f(x_{n-1},x_{n-2},\ldots, x_{n-(k-1)}).
\]
\end{enumerate}
Then $\lim_{n\to +\infty} x_n=\overline{x}$.
\end{lem}

The straightforward proof of Lemma \ref{ThA01Lemma} is found in \cite{KulenovicLadas}, 
see their Theorem 1.4.8 (for $k=3$), Theorem A.0.1 (for $k=4$) or Theorem A.0.9 
(for general $k$)---the idea of the proof is the same in the easiest case when $f$ is non-decreasing in each of its arguments.
Notice that in \cite[Theorem 1.4.8] {KulenovicLadas} one should have ``The equation $f(x,x)=x$ has a unique solution in $[a,b]$'' instead of ``\ldots\, a unique positive solution''. Now we give the proof of Theorem \ref{attractivitytheorem}.
\begin{proof}[The proof of Theorem \ref{attractivitytheorem}] For some $a\le b$ (to be specified soon) we set 
\begin{equation}\label{section8fdef}
f(z_1, z_2, \ldots, z_{k-1}):=\frac{\sum_{j=1}^{k-1}z_{j}}{A+\sum_{j=1}^{k-1}z_{j}}
\end{equation}
with $z_j\in [a,b]$ ($j=1, 2, \ldots, k-1$). Then 
\begin{equation}\label{section8fpartial}
(\partial_j f)(z_1, z_2, \ldots, z_{k-1})=\frac{A}{\left(A+\sum_{j=1}^{k-1}z_{j}\right)^2}>0,
\end{equation}
hence $f$ is non-decreasing in each of its arguments (and trivially continuous). Notice that due to
\eqref{taudef} we have $\tau_n\in (0,1)$ for any $n\ge k$, so by shifting the indices we
can assume that $\tau_j\in (0,1)$ for $1\le j\le k-1$, and  
$\tau_n=f(\tau_{n-1},\tau_{n-2},\ldots, \tau_{n-(k-1)})$ for $n\ge k$. We distinguish two cases.
\begin{enumerate}
\item \textit{The case $k-1\le A$.}  Then $f$ maps $[a,b]^{k-1}$ to $[a,b]$ with $a:=0$ and $b:=1$. 
Now for any $\overline{x}\in [0,1]$ we have
\begin{equation}\label{factorization}
\overline{x}-f(\overline{x},\overline{x},\ldots, \overline{x})\equiv \frac{\overline{x} [A-(k-1)+(k-1) \overline{x}]}{A+(k-1) \overline{x}}=0
\end{equation}
precisely if $\overline{x}=0$, so Lemma \ref{ThA01Lemma} yields $\lim_{n\to +\infty} \tau_n=0$.
\item \textit{The case $0<A<k-1$.}  We set $\tau^*:=\min_{1\le j\le k-1} \tau_j$. Then $\tau^*>0$
and $\tau_j\in [\tau^*,1]$ for $1\le j \le k-1$.
\begin{enumerate}
\item  \textit{The case $\tau^*\ge \frac{k-1-A}{k-1}$.} We choose $a:=\frac{k-1-A}{k-1}<1=:b$, and notice that $f(a,a,\ldots, a)=a$ and $f(b, b, \ldots, b)<b$. By also using the non-decreasing property of $f$ in each of its arguments we obtain that $f$ maps $[a,b]^{k-1}$ to $[a,b]$. Now for any $\overline{x}\in [a,b]$ the equality 
\eqref{factorization} holds 
if and only if $\overline{x}=a$, so Lemma \ref{ThA01Lemma} yields 
$\lim_{n\to +\infty} \tau_n=\frac{k-1-A}{k-1}$. 
\item  \textit{The case $0<\tau^*< \frac{k-1-A}{k-1}$.} This time we choose $a:=\tau^*<1=:b$. Since 
now
\[
a<f(a,a,\ldots, a) \Longleftrightarrow \tau^*<f(\tau^*,\tau^*,\ldots, \tau^*) \Longleftrightarrow
\tau^*< \frac{k-1-A}{k-1},
\]
we have just as before that $f$ maps $[a,b]^{k-1}$ to $[a,b]$. For any $\overline{x}\in [a,b]$ we have
\eqref{factorization} 
precisely if $\overline{x}=\frac{k-1-A}{k-1}$, therefore we can use Lemma \ref{ThA01Lemma} again to get
$\lim_{n\to +\infty} \tau_n=\frac{k-1-A}{k-1}$. 
\end{enumerate}
\end{enumerate}
\end{proof}
\begin{exa}
The sequence $(\tau_n)_{n\ge 1}$ defined by \eqref{taudef} can have long, non-monotonic starting slices. Consider, for example, the case $k=4$ with 
\[\tau_1:=1,\quad \tau_2:=\frac{1}{200},\quad \tau_3:=\frac{95638788642}{100000000000},
\]
and
\[
\tau_n:=\frac{\tau_{n-1}+\tau_{n-2}+\tau_{n-3}}{1+\tau_{n-1}+\tau_{n-2}+\tau_{n-3}}\quad\quad \text{for } n\ge 4.
\]
Then the consecutive monotone non-increasing subsequences of $\tau_{n}$ for 
$1\le n\le 1000$ has lengths
\ifpaper
	\begin{multline*}
	(2, 3, 2,1,2,1,2,1,2,1,2,3,3,3,3,2,1, 2, 1, 2, 1, \\ 2,1,2,3, 3, 3, 3, 2,1,2,1,2,1,2,3,3,3,3, 3, 917).
	\end{multline*}
\else
	\begin{align*}
	(2, 3, 2,1,2,1,2,1,2,1,2,3,3,3,3,2,1, 2, 1, 2, 1, 2,1,2,3, 3, 3, 3, 2,1,2,1,2,1,2,3,3,3,3, 3, 917).
	\end{align*}
\fi

\end{exa}

\subsection{The proof of Theorem \ref{hnconvergencetheorem}}\label{proofofTh2section}

\begin{proof}
We prove the theorem for $k=3$ first. 
We define 
two sequences
\begin{equation}\label{ineq27}
h_n^-:=\frac{h_{n-2}^{-}+h_{n-1}^{-}}{h_{n-2}^{-}+h_{n-1}^{-} + \mu^-}\cdot\mu^-,\quad
h_n^{+}:=\frac{h_{n-2}^{+}+h_{n-1}^{+}}{h_{n-2}^{+}+h_{n-1}^{+} + \mu^{+}}\cdot\mu^{+},
\quad
h_1^\pm:=h_1,\  h_2^\pm:=h_2,
\end{equation}
and their scaled counterparts
$
\tau_n^-:=h_n^-/\mu^-$,
$\tau_n^{+}:=h_n^{+}/\mu^{+}$
($n\ge 1$).
Then $\tau_n^-$ and $\tau_n^{+}$ satisfy
\[
\tau_n^- = \frac{\tau_{n-2}^-+ \tau_{n-1}^-}{\tau_{n-2}^- + \tau_{n-1}^- + 1}, \quad 
\tau_n^{+} = \frac{\tau_{n-2}^{+} + \tau_{n-1}^{+}}{\tau_{n-2}^{+} + \tau_{n-1}^{+} + 1},\quad \tau_1^\pm>0, \ \tau_2^\pm>0.
\]
By applying Theorem~\ref{attractivitytheorem} with $k=3$ and $A=1$ we see that 
$\tau_n^-\to {1}/{2}$ hence $h_n^-\to {\mu^-}/{2}$ as $n\to+\infty$.
Similarly, we get $h_n^+\to {\mu^+}/{2}$. We now define 
\begin{equation}\label{ftildefunctiondef}
(0,+\infty)^3\ni (a,x,y)\mapsto\widetilde{f}(a,x,y):=a\cdot\frac{x+y}{a+x+y}
\end{equation}
(cf. \eqref{section8fdef}). It is elementary to see that for any $(a,x,y)\in (0,+\infty)^3$ we have
\begin{equation}\label{ftildepospartder}
\partial_1 \widetilde{f}(a,x,y)=\frac{(x+y)^2}{(a+x+y)^2}>0, \quad \partial_2 \widetilde{f}(a,x,y)=\partial_3 \widetilde{f}(a,x,y)=
\frac{a^2}{(a+x+y)^2}>0
\end{equation}
(cf. \eqref{section8fpartial}, and notice that the function
$a\mapsto \widetilde{f}(a,x,y)/a$, for example, would be monotone \textit{decreasing}).
Clearly, for $n=1, 2$ we have 
\begin{equation}\label{hnsandwich}
h_n^-\le h_n \le h_n^{+},
\end{equation}
so we can suppose that \eqref{hnsandwich} has already been proved up to some
$n\ge 2$. Then by repeatedly using the inequality $\mu^-\le\mu_n\le\mu^{+}$
(implied by the assumption \eqref{murestriction}), \eqref{ftildepospartder} and
\eqref{hnsandwich}, we obtain
\[
h_{n+1}^-\equiv \frac{h_{n-1}^-+h_{n}^-}{h_{n-1}^-+h_{n}^- + \mu^-} \cdot\mu^-\le 
\frac{h_{n-1}+h_{n}^-}{h_{n-1}+h_{n}^- + \mu^-} \cdot\mu^-\le
\frac{h_{n-1}+h_{n}}{h_{n-1}+h_{n} + \mu^-} \cdot\mu^-\le \]
\[
\frac{h_{n-1}+h_{n}}{h_{n-1}+h_{n} + \mu_{n+1}} \cdot\mu_{n+1}\equiv h_{n+1}\le
\]
\[
\frac{h_{n-1}+h_{n}}{h_{n-1}+h_{n} + \mu^+} \cdot\mu^+\le
\frac{h_{n-1}+h_{n}^+}{h_{n-1}+h_{n}^+ + \mu^+} \cdot\mu^+\le
\frac{h_{n-1}^++h_{n}^+}{h_{n-1}^++h_{n}^+ + \mu^+} \cdot\mu^+\equiv h_{n+1}^+.
\]
This shows the validity of \eqref{hnsandwich} for all $n\ge 1$ by induction.
By taking $\liminf$ and $\limsup$ in \eqref{hnsandwich}, Theorem \ref{hnconvergencetheorem} for $k=3$ is proved. 

The proof of Theorem~\ref{hnconvergencetheorem} in the general case requires only formal  modifications of the argument given above: Theorem~\ref{attractivitytheorem} with a general $k\ge 3$ and with $A=1$ implies that for the corresponding sequences we have
$\tau_n^\pm\to \frac{k-2}{k-1}$ as $n\to +\infty$, and the corresponding 
function $\widetilde{f}\colon (0,+\infty)^k\to (0,+\infty)$ is increasing in each of its arguments.
\end{proof}

Figure \ref{hminushplus} gives a graphical illustration of Theorem \ref{hnconvergencetheorem} for 
$k=3$,
using a hypothetical sequence of values for $\mu_n$.

\begin{figure}
\begin{center}
\includegraphics[width=0.66\textwidth]{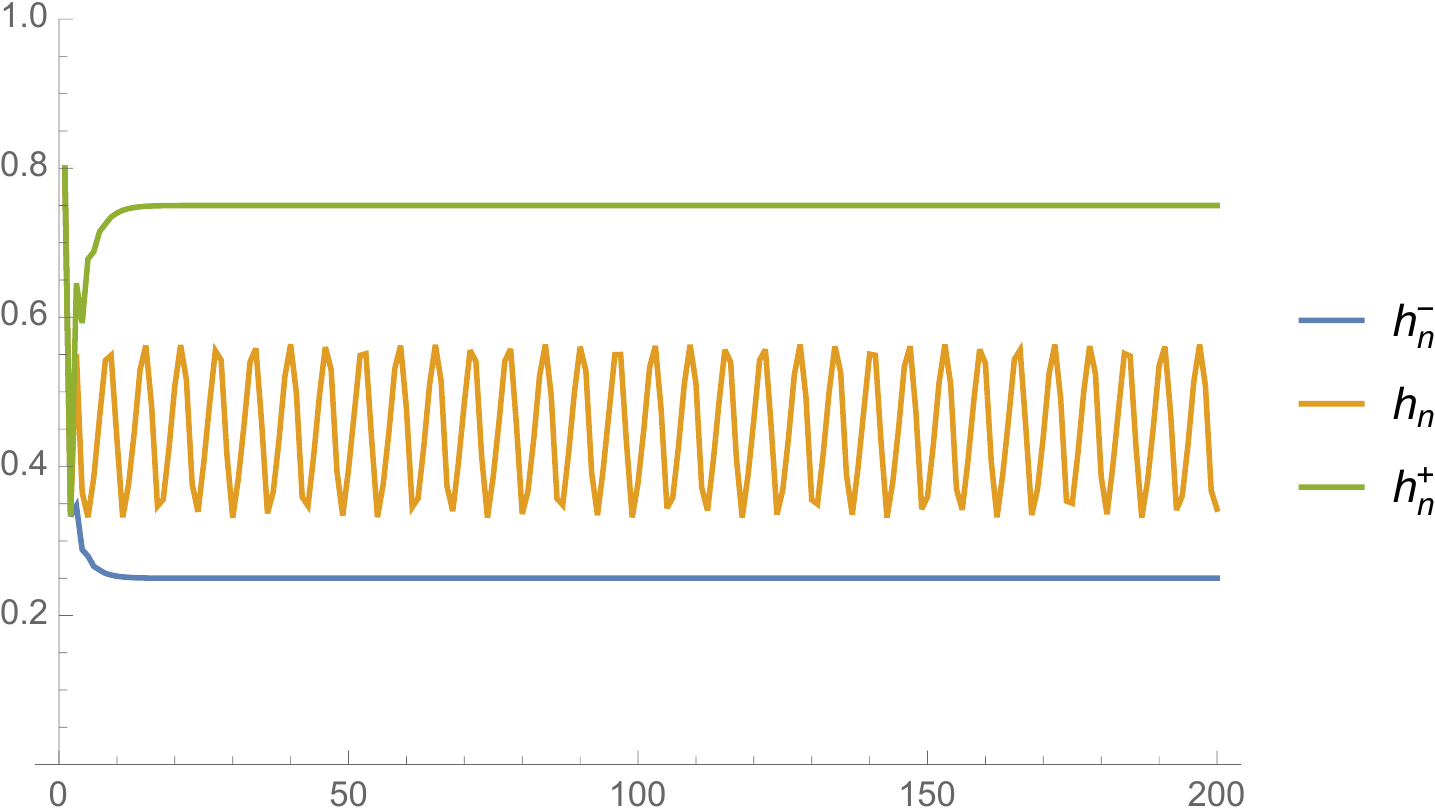}
\caption{The first 200 terms of the sequences $h_n^-$, $h_n$ and $h_n^+$ with
$\mu_n:=1+\sin(n)/2$, $\mu^-:=1/2$ and $\mu^+:=3/2$, see \eqref{greedydef} and \eqref{ineq27}.
\label{hminushplus}}
\end{center}
\end{figure}

\subsection{The proof of Theorem \ref{hnconvergencetheorem3rdorder}}\label{proofofTh3section}

\begin{proof}

\textbf{Step 1.} Initially we suppose that some values of $\varrho>0$ and $0<\rhoFE\le 1$ have already been chosen; we will make an actual choice for them in Step 5 so that the inductive
argument given below becomes valid. First let us set $n=k$.

\textbf{Step 2.} We know from \eqref{hjrestriction} that
\begin{equation}\label{label46}
h_m\le \varrho\cdot\hFE(u_m) \quad \text{for }\  m=1, 2, \ldots, n-1.
\end{equation}
We will prove that
\begin{equation}\label{label44}
\sum_{j=1}^{k-1} h_{n-j} \le \sqrt{8}\, \mu_n
\end{equation}
and
\begin{equation}\label{label45}
h_n\le \varrho\cdot\hFE(u_n).
\end{equation}

\textbf{Step 3.} Applying \eqref{label46} and \eqref{hferestriction} repeatedly we have
\begin{equation}\label{label47}
\sum_{j=1}^{k-1} h_{n-j} \le \varrho \cdot \sum_{j=1}^{k-1} \hFE(u_{n-j}) \le
\varrho \cdot \sum_{j=1}^{k-1} \frac{\hFE(u_{n-1})}{(\rhoFE)^{j-1}} =
\hFE(u_{n-1})\cdot \sum_{j=1}^{k-1} \frac{\varrho}{(\rhoFE)^{j-1}}. 
\end{equation}
On the other hand, the definition of $\mu_n$ in \eqref{generalmundef}, the repeated application of \eqref{hferestriction}, and  $0<\rhoFE\le 1$ imply that
\begin{equation}\label{label48}
\sqrt{8}\, \mu_n\ge \sqrt{8}\min_{0 \le j \le k-1} \left((\rhoFE)^{j}\cdot\hFE(u_{n-1})\right)
=\hFE(u_{n-1})\cdot \sqrt{8}\, (\rhoFE)^{k-1}.
\end{equation}
By comparing the right-hand sides of \eqref{label47} and \eqref{label48} 
after a division by $\hFE(u_{n-1})\ge\mu^->0$, we get that if $\varrho$ and $\rhoFE$
are chosen such that
\begin{equation}\label{label49suff1}
\sum_{j=1}^{k-1} \frac{\varrho}{(\rhoFE)^{j-1}}\le \sqrt{8}\, (\rhoFE)^{k-1},\quad\quad 
\varrho>0, \quad 0<\rhoFE\le 1,
\end{equation}
then \eqref{label44} holds for this particular $n$ value.

\textbf{Step 4.} In order to show \eqref{label45}, we first notice that
 the function  
\[
(0,+\infty)^2\ni (z,a)\mapsto \widetilde{f}(z,a):=\frac{z}{z+2a}\cdot a
\]
is increasing in each of its arguments
 (cf.~\eqref{ftildefunctiondef}-\eqref{ftildepospartder}). This monotonicity property, the 
definition of $h_n$ in \eqref{k45hndef}, \eqref{label47}, and the inequality 
$\mu_n\le \hFE(u_{n-1})$ yield that

\ifpaper
	\begin{align*}
		h_n &=\frac{\sum_{j=1}^{k-1} h_{n-j}}{\left(\sum_{j=1}^{k-1} h_{n-j} \right)+2\mu_n}\cdot\mu_n \\
		&\le \frac{\hFE(u_{n-1})\cdot \sum_{j=1}^{k-1} \frac{\varrho}{(\rhoFE)^{j-1}}
		}{\left(\hFE(u_{n-1})\cdot \sum_{j=1}^{k-1} \frac{\varrho}{(\rhoFE)^{j-1}}
		\right)+2\hFE(u_{n-1})}\cdot\hFE(u_{n-1}) \\
		&= \frac{\sum_{j=1}^{k-1} \frac{\varrho}{(\rhoFE)^{j-1}}
		}{2+\sum_{j=1}^{k-1} \frac{\varrho}{(\rhoFE)^{j-1}}}\cdot\hFE(u_{n-1}).
	\end{align*}
\else
	\begin{align*}
		h_n =\frac{\sum_{j=1}^{k-1} h_{n-j}}{\left(\sum_{j=1}^{k-1} h_{n-j} \right)+2\mu_n}\cdot\mu_n
		&\le \frac{\hFE(u_{n-1})\cdot \sum_{j=1}^{k-1} \frac{\varrho}{(\rhoFE)^{j-1}}
		}{\left(\hFE(u_{n-1})\cdot \sum_{j=1}^{k-1} \frac{\varrho}{(\rhoFE)^{j-1}}
		\right)+2\hFE(u_{n-1})}\cdot\hFE(u_{n-1}) \\
		&= \frac{\sum_{j=1}^{k-1} \frac{\varrho}{(\rhoFE)^{j-1}}
		}{2+\sum_{j=1}^{k-1} \frac{\varrho}{(\rhoFE)^{j-1}}}\cdot\hFE(u_{n-1}).
	\end{align*}
\fi
On the other hand, from \eqref{hferestriction} we see that
$\varrho\cdot\hFE(u_n)\ge \varrho\cdot\rhoFE\cdot\hFE(u_{n-1})$.
Therefore if $\varrho$ and $\rhoFE$
are chosen such that
\begin{equation}\label{label52suff2}
\frac{\sum_{j=1}^{k-1} \frac{\varrho}{(\rhoFE)^{j-1}}
}{2+\sum_{j=1}^{k-1} \frac{\varrho}{(\rhoFE)^{j-1}}}\le \varrho\cdot\rhoFE,
\end{equation}
then \eqref{label45} holds for the actual $n$ value.

\textbf{Step 5.} So \eqref{label44}-\eqref{label45} will be proved as soon as we have found
some $\varrho$ and $\rhoFE$ satisfying \eqref{label49suff1} and \eqref{label52suff2}. Figure
\ref{doubletrianlges} depicts the solution set of this system of inequalities 
\eqref{label49suff1}-\eqref{label52suff2} in the variables $(\varrho,\rhoFE)$. 
\begin{figure}
\begin{center}
\includegraphics[width=\textwidth]{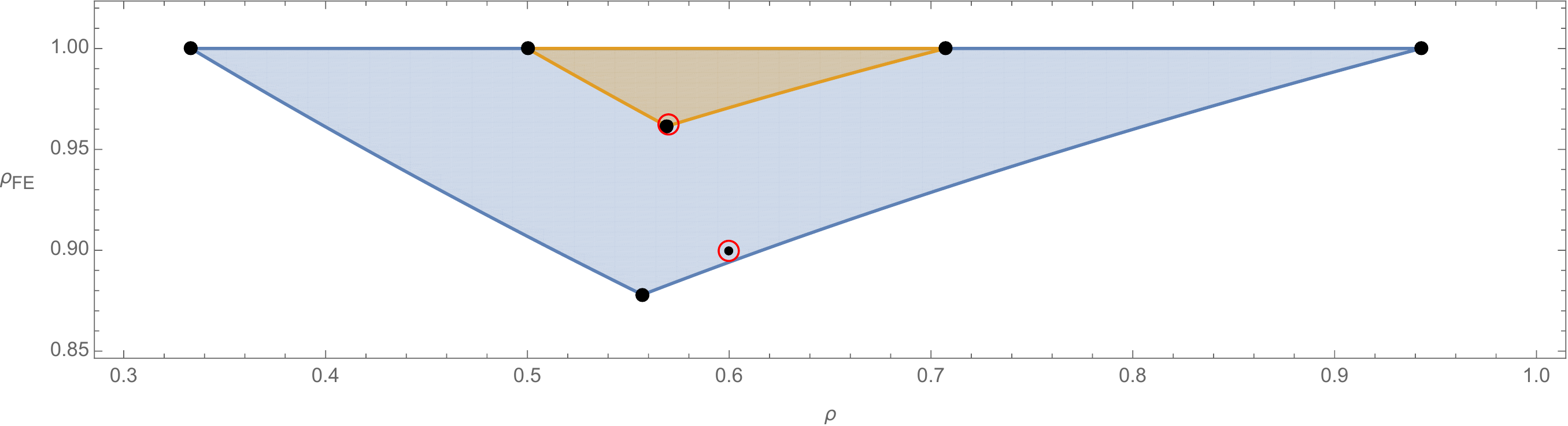}
\caption{The solution set of the inequalities \eqref{label49suff1}-\eqref{label52suff2}. 
In the $k=4$ case, the blue region has vertices at $(1/3, 1)$, $(\sqrt{8}/3, 1)$,  
$\approx(0.557, 0.878)$. In the $k=5$ case, the orange region has vertices at $(1/2,1)$, 
$(1/\sqrt{2},1)$, $\approx(0.569, 0.9615)$. The points corresponding to \eqref{particularrhovalues} have been circled.
\label{doubletrianlges}}
\end{center}
\end{figure}
The choice made in \eqref{particularrhovalues} is a simple rational pair; clearly,
one could for example relax the assumption on $\varrho$ (by choosing it larger), but then condition \eqref{hferestriction} would in general become more stringent.

\textbf{Step 6.}  According to the discussion preceding Theorem \ref{hnconvergencetheorem3rdorder}, on the one hand we have  
$\Omega_{k-1,n}>2$ hence $\sspcoeff_n>0$. On the other hand, \eqref{label44} guarantees   
$h_n=\sspcoeff_n \mu_n>0$ and this is the 
maximum value of $h_n$ preserving the SSP property.

\textbf{Step 7.} Now we repeat Steps 2-6 inductively for each $n\ge k+1$ 
(Step 5 is no longer needed since $(\varrho,\rhoFE)$ have already been given some particular values).
The range of $m$ in the induction hypothesis \eqref{label46} is extended step-by-step
by \eqref{label45}.

\textbf{Step 8.} Finally, to prove \eqref{3ndorderinfsup}, we make use of the fact that
\[(0,+\infty)^k\ni (a,z_1,z_2,\ldots ,z_{k-1})\mapsto \bar{f}(a,z_1,z_2,\ldots ,z_{k-1}):=a\cdot\frac{\sum_{j=1}^{k-1} z_{j}}{\left(\sum_{j=1}^{k-1} z_{j}\right)+2a}
\]
is increasing in each of its arguments (cf.~\eqref{ftildefunctiondef}-\eqref{ftildepospartder}), and
repeat the steps presented in Section \ref{proofofTh2section}: 
we define the corresponding auxiliary sequences $h_n^\pm$ and $\tau_n^\pm$ and apply 
Theorem \ref{attractivitytheorem} with $k\in\{4,5\}$ and $A=2$ to show that  
$\tau_n^\pm\to \frac{k-3}{k-1}$ as $n\to +\infty$. 
\end{proof}

\subsection*{Acknowledgement}
The authors would like to thank the anonymous referees for their suggestions, 
which have improved the presentation of the material.

%%% Bibliography %%%
\ifpaper
	\bibliographystyle{siam}
	\bibliography{ssp_lmm_vss}
\else
	\sloppy
	\printbibliography
\fi

\end{document}